\documentclass[11pt]{article}

\usepackage[margin=1in]{geometry}
\usepackage[T1]{fontenc}
\usepackage[utf8]{inputenc}
\usepackage{lmodern}
\usepackage{microtype}
\usepackage{amsmath,amssymb,amsthm,mathtools}
\usepackage{booktabs}
\usepackage{graphicx}
\usepackage{enumitem}
\usepackage[hidelinks]{hyperref}
\hypersetup{
  pdftitle={Symmetry reductions and recurrence degrees for banded Toeplitz determinants and permanents},
  pdfauthor={Max A. Alekseyev and Dmitry I. Khomovsky},
  pdfsubject={Symmetry reductions for recurrences of banded Toeplitz determinants and permanents},
  pdfkeywords={Toeplitz matrix, determinant, permanent, linear recurrence, Polya signing, circulant matrix}
}
\usepackage{doi}
\usepackage{authblk}
\usepackage[backend=biber,citestyle=numeric-comp,bibstyle=ieee,sorting=none,minbibnames=5,maxbibnames=8,giveninits=true]{biblatex}
\newtheorem{theorem}{Theorem}[section]
\newtheorem{proposition}[theorem]{Proposition}
\newtheorem{lemma}[theorem]{Lemma}
\newtheorem{corollary}[theorem]{Corollary}
\newtheorem{example}[theorem]{Example}
\theoremstyle{remark}
\newtheorem{remark}[theorem]{Remark}

\DeclareMathOperator{\perm}{perm}

\newcommand{\eps}{\varepsilon}
\newcommand{\defeq}{:=}

\title{Symmetry reductions and recurrence degrees for banded Toeplitz determinants and permanents}
\author[1]{Max A. Alekseyev}
\author[2]{Dmitry I. Khomovsky}
\affil[1]{\small The George Washington University, Washington, DC, USA. Email: \href{mailto:maxal@gwu.edu}{maxal@gwu.edu}}
\affil[2]{\small Email: \href{mailto:khomovskij@physics.msu.ru}{khomovskij@physics.msu.ru}}
\date{}

\begin{document}
\maketitle

\begin{abstract}
This paper studies symmetry-induced reductions of scalar recurrence complexity for
balanced banded Toeplitz determinants and permanents.  Three mechanisms
emerge: determinant state and spectral compression, exceptional
permanent--determinant conversion, and symmetries of permanent state spaces.

For symmetric determinants, straightening reduces the row-column states from
$\binom{2m}{m}$ to the Catalan number $C_{m+1}$; primitive symplectic weight
compression then leaves $(3^m+1)/2$ distinct autonomous modes.  Skew symmetry
has a parallel compound/Hodge explanation: the middle exterior representation
splits into two Hodge halves of dimension $\binom{2m}{m}/2$, the one-step
compound transfer exchanges the halves, and its two-step restriction has
$3^{m-1}$ generic ternary modes.  Thus the full skew bound is
$2\cdot3^{m-1}$.  Widom--Hankel arguments
prove generic scalar minimality in both symmetry classes, and every even skew
Toeplitz determinant admits an explicit half-size square factorization.

For the zero-diagonal pentadiagonal support, a P\'olya--Kasteleyn signing
converts every permanent to a determinant.  Among consecutive zero-diagonal
two-sided bands, universal entrywise conversion---and separately Toeplitz
diagonal-wise conversion---occurs only in the Hessenberg families and this
pentadiagonal case.  Paired renewal identities recover all restored-diagonal
determinant and permanent layers.

For permanents, transposition gives the open symmetric bound
$(\binom{2m}{m}+2^m)/2$, while cyclic defect-sector pairing gives
$(4^m+\binom{2m}{m})/2$.  Skew-symmetry forces odd-size vanishing and
corresponding even-subsequence bounds.  In semibandwidth two the open
symmetric bound is generically sharp; higher-semibandwidth permanent
minimality is left separate from the symmetry reductions proved here.
\end{abstract}



\section{Introduction}

Let
\[
 A_n=(a_{j-i})_{i,j=1}^n,
 \qquad
 a(z)=\sum_{s=-m}^{m}a_s z^s,
\]
be the leading $n\times n$ section of a balanced banded Toeplitz matrix with lower and upper semibandwidth $m$.  We write
\[
 D_n\defeq\det A_n,\qquad P_n\defeq\perm A_n,
\]
and reserve \emph{bandwidth} for the total number $2m+1$ of potentially nonzero diagonals.
The present paper is the symmetry-focused follow-up to
\cite{AlekseyevKhomovskyRecurrences2026}.  The companion paper develops the
two constructive Laplace-expansion mechanisms, identifies the exact
row-column state space, proves the unrestricted determinant degree
$\binom{2m}{m}$ to be generically minimal, and gives the corresponding
permanent upper bound.  Here we keep that general machinery in the
background and ask what further structure becomes available after balance,
reflection, or sign constraints are imposed.

A point important for what follows is that the determinant reductions below
are not visible from the \emph{naive} row-column state count.  Under the
symmetric specialization the unreduced transfer still has $\binom{2m}{m}$
boundary states, and quotienting only by transposition would give the orbit
count
\[
 \frac12\left(\binom{2m}{m}+2^m\right),
\]
which is exactly the state-space bound that appears later for symmetric
permanents.  Determinant signs impose additional straightening relations among
symmetric boundary minors, reducing the row-column construction to the Catalan
number $C_{m+1}$.  We then identify this Catalan state space with the primitive
middle exterior module for the reciprocal companion transfer.  Its weight
multiplicities are Catalan numbers, but a fixed autonomous transfer acts by one scalar
on each weight; the number of admissible weights is exactly $(3^m+1)/2$.
Thus the complete upper-bound reduction
\[
 \binom{2m}{m}\longrightarrow C_{m+1}
 \longrightarrow\frac{3^m+1}{2}
\]
has a transfer-theoretic explanation as well as the direct Widom root-product
explanation.  The latter remains the more economical route to the modes and,
together with the Hankel argument below, proves generic minimality of the
principal determinant sequence.  For permanents, by contrast, the reductions
proved below are visible directly as quotients of the finite boundary-state
space.

Three mechanisms should be distinguished.  The first is \emph{determinant state and mode compression}.  Widom's
characteristic modes are fixed-cardinality products of the roots of
$z^ma(z)$.  Symmetry or skew-symmetry pairs those roots reciprocally, so many
products coalesce.  In the symmetric case the same phenomenon can also be
seen constructively: Laplace signs straighten boundary minors to a Catalan
state space, and its primitive symplectic weight decomposition then collapses
the autonomous transfer to $(3^m+1)/2$ distinct eigenvalues.  Skew symmetry
admits a parallel compound/Hodge reduction.  In the middle exterior
realization, the transfer space splits into two Hodge halves of dimension
$\binom{2m}{m}/2$; the one-step compound transfer exchanges the halves, while
its square has $3^{m-1}$ generic ternary modes.  Fixed-ambient skew boundary
minors obey the expected transpose-sign identities, but the normalized
row-column vector uses different ambient section sizes across levels, so those
identities do not define a levelwise half-dimensional quotient of the
one-step row-column transfer.  Hence the full one-step bound is
$2\cdot3^{m-1}$.  Widom gives the same mode counts directly, and a
Hankel--Vandermonde argument shows that they are generically exact for the
scalar determinant sequences.  The skew case also carries the classical
Pfaffian square, which for Toeplitz sections can be sharpened to a half-size
Toeplitz-minus-Hankel factorization.
Related reciprocal-root, symmetric-polynomial, and low-bandwidth Chebyshev
structures occur in the Toeplitz literature, including
\cite{Zakr,Alexandersson2012,AlexanderssonEtAl2021,Elouafi2011,Elouafi2018}.

The second mechanism is \emph{exact conversion}.  P\'olya's signing problem
and its planar/Pfaffian descendants are classical
\cite{Polya1913,Kasteleyn1967,Kuperberg1998,RobertsonSeymourThomas1999}.  In
the zero-diagonal pentadiagonal band we obtain an entrywise signing that
converts every permanent to a determinant, even without Toeplitz constancy;
on the Toeplitz locus it is gauge-equivalent to a fourth-root-of-unity
diagonal substitution.  Paired fixed-point renewal identities reconstruct
the arbitrary-main-diagonal permanent and determinant from their zero-diagonal
sequences, with a single sign distinguishing the unique three-vertex
configuration in which a fixed point is crossed by a transposition.  We then
classify consecutive zero-diagonal
two-sided bands at two different levels.
For the full coordinate subspace, even arbitrary nonzero entrywise Schur
multipliers work only in the Hessenberg cases and the pentadiagonal case; a
six-matching obstruction already rules out the $(3,2)$ support at order six.
On the Toeplitz locus, imposing instead that the multipliers be constant along
diagonals gives exactly the same list.  This separates a genuine low-width
identity between determinant and permanent from the state-space reductions
that persist at arbitrary semibandwidth, and it complements earlier results
on Hessenberg-type matrices and convertible subspaces
\cite{DaFonseca2011,DaCruzEtAl2017,HwangKimSong1996}.

The third mechanism is \emph{state-space symmetry} for permanents.  No Widom
root-product formula is required: transposition exchanges the two subset
labels of the balanced row-column states.  For cyclic closure, a column shift
collects all wrap-around entries into a fixed corner defect.  Resolving the
defect splits the permanent into contributions whose clean central bands have
semibandwidths $(2m-r,r)$, and transposition pairs the sectors $r$ and
$2m-r$.  The resulting state counts give the open and cyclic bounds
\[
 \frac12\left(\binom{2m}{m}+2^m\right),
 \qquad
 \frac12\left(4^m+\binom{2m}{m}\right).
\]
These are unconditional annihilator bounds for every $m$.  The open
pentadiagonal case can be settled symbolically, including its exceptional
degree-drop locus; generic minimality in arbitrary semibandwidth is not
needed for, and is not pursued as part of, the symmetry reductions proved
here.

The paper is organized as follows.  Section~\ref{sec:framework} recalls only
the pieces of the preceding recurrence paper needed here.
Section~\ref{sec:det-symmetry} treats determinant reciprocity, symmetry
reductions, skew Toeplitz factorization, and generic sharpness.
Section~\ref{sec:conversion} isolates the distinct low-semibandwidth
permanent--determinant conversion mechanism.  Section~\ref{sec:permanent-symmetry}
develops the open permanent symmetry quotients.  Section~\ref{sec:cyclic-symmetry}
first gives a circulant-completion counterpart to the classical fixed-size
reductions for banded Toeplitz determinants and shows how it recovers the
Widom expansion, then treats cyclic determinant symmetry and the cyclic
permanent bounds.

\section{Framework recalled from the companion recurrence paper}\label{sec:framework}

We recall the minimum notation needed from \cite{AlekseyevKhomovskyRecurrences2026}.  Put
\[
 q(z)=z^ma(z)=a_m\prod_{j=1}^{2m}(z-z_j).
\]
When the roots are simple, Widom's formula gives
\begin{equation}\label{eq:widomform}
 D_n=\sum_{|J|=m}C_JW_J^n,
 \qquad
 W_J=(-1)^m a_m\prod_{j\in J}z_j,
\end{equation}
and hence the degree-$d$ annihilator
\begin{equation}\label{eq:widompoly}
 \chi(t)=\prod_{|J|=m}(t-W_J),
 \qquad d=\binom{2m}{m}.
\end{equation}
The coefficients extend polynomially through root collisions.  As in the companion paper, ``generic minimal degree'' means the minimal homogeneous constant-coefficient recurrence degree on a nonempty Zariski-open subset of the relevant genuine-band parameter space.

For permanents, let $Q_m^+$ denote the signless row-column transfer of the companion paper.  Its normalized states at exchange level $j$ are indexed by pairs
\[
 A,B\in\binom{[m]}j,\qquad 0\le j\le m.
\]
Writing $\ell=j+1$, $B=\{b_1<\cdots<b_j\}$, and
\[
 \widehat A=\begin{cases}
 A,&1\notin B,\\
 (A\setminus\{1\})\cup\bigl(\{m+1\}\text{ if }1\in A\bigr),&1\in B,
 \end{cases}
\]
with $\widehat A=\{\widehat a_1<\cdots<\widehat a_j\}$, define
\begin{equation}\label{eq:sigmaAB}
 \Sigma(A,B)=
 \left(
  (\ell,\ell-b_1,\ldots,\ell-b_j),
  (\ell,\ell-\widehat a_1,\ldots,\ell-\widehat a_j)
 \right).
\end{equation}
For a normalized signature of this form, with increasing rearrangements $x_1<\cdots<x_j$ and $y_1<\cdots<y_j$ of the nonleading row and column entries, the inverse labels are
\begin{equation}\label{eq:AB-from-signature}
 B_\sigma=\{\ell-x_t:1\le t\le j\},\qquad
 A_\sigma=\{\langle\ell-y_t\rangle_m:1\le t\le j\},
\end{equation}
where $\langle r\rangle_m$ denotes the representative in $[m]$ modulo $m$.

\begin{proposition}[Balanced row-column state space, recalled from \cite{AlekseyevKhomovskyRecurrences2026}]\label{prop:state-recall}
The balanced determinant and permanent row-column transfers reach exactly the states $\Sigma(A,B)$ above.  Their level-$j$ state count is $\binom mj^2$, and the total transfer dimension is
\begin{equation}\label{eq:statecount-recalled}
 \sum_{j=0}^m\binom mj^2=\binom{2m}{m}.
\end{equation}
In particular, the unrestricted balanced permanent sequence has a homogeneous constant-coefficient annihilator of degree at most $\binom{2m}{m}$.
\end{proposition}

\begin{proof}
This is the balanced specialization of the exact signature classification and state-count theorem in \cite{AlekseyevKhomovskyRecurrences2026}.  The formulas \eqref{eq:sigmaAB} and \eqref{eq:AB-from-signature} are recalled because the involutions below act transparently on these labels.
\end{proof}

For later specializations we also use the Lucas-type polynomials
\[
 U_0(P,Q)=0,\qquad U_1(P,Q)=1,\qquad
 U_{r+2}(P,Q)=P\,U_{r+1}(P,Q)-Q\,U_r(P,Q).
\]
Finally, $T_n(c_{-\ell},\ldots,c_u)$ denotes the $n\times n$ Toeplitz matrix whose diagonal of offset $s=j-i$ has constant entry $c_s$.

\section{Determinant symmetry and generic sharpness}\label{sec:det-symmetry}

We begin with the determinant side, where the root-product description makes the effect of balanced symmetry completely explicit.  Let $m_1=m_2=m$, put
\[
 d=\binom{2m}{m},
 \qquad
 \varpi=a_{-m}a_m,
\]
and let $\chi(t)=t^d-r_1t^{d-1}-\cdots-r_d$ be the degree-$d$ Widom
annihilator recalled in \eqref{eq:widompoly}.

\subsection{Reciprocity}

\begin{theorem}[Reciprocal full recurrence]\label{thm:reciprocal}
For a balanced $(m,m)$-banded Toeplitz determinant sequence, the coefficients
of the degree-$d$ Widom annihilator satisfy
\begin{equation}\label{eq:recipcoeff}
 r_i=\varpi^{\,i-d/2}r_{d-i},
 \qquad d/2\le i\le d,
\end{equation}
where $r_0\defeq-1$.  In particular,
\[
 r_d=-\varpi^{d/2}.
\]
\end{theorem}

\begin{proof}
Let $z_1,\ldots,z_{2m}$ be the roots of $q(z)=z^ma(z)$.  For $|J|=m$, put
\[
 W_J=(-1)^m a_m\prod_{j\in J}z_j.
\]
Since $q$ has constant term $a_{-m}$ and leading coefficient $a_m$,
\[
 \prod_{j=1}^{2m}z_j=\frac{a_{-m}}{a_m}.
\]
If $J^c$ is the complementary $m$-subset, then
\[
 W_JW_{J^c}=a_m^2\prod_{j=1}^{2m}z_j=a_{-m}a_m=\varpi.
\]
Thus the roots of $\chi$ are invariant under $W\mapsto \varpi/W$.  Because $d$
is even,
\[
 t^d\chi(\varpi/t)=\varpi^{d/2}\chi(t).
\]
Comparing the coefficients of $t^{d-i}$ gives \eqref{eq:recipcoeff}.
Distinctness of the roots is not needed for this complement argument: if
$q$ has repeated roots, list them with multiplicity; then
$J\leftrightarrow J^c$ gives the same pairing of the factors $W_J$.
\end{proof}

This theorem explains the reciprocal pattern already visible in Sweet's
pentadiagonal recurrence \cite{Sweet}, which is recovered constructively in the companion paper \cite{AlekseyevKhomovskyRecurrences2026}, and reduces by roughly half the number of independent coefficients that need to be computed in the balanced case.

\begin{example}[The balanced $(3,3)$ recurrence]\label{ex:33-reciprocity}
For $m=3$ the unrestricted determinant transfer has dimension
$d=\binom63=20$.  The companion recurrence paper
\cite{AlekseyevKhomovskyRecurrences2026} displays the explicit sparse
$20\times20$ row-column transfer $Q_{3,3}$.  Write
\[
 \det(tI-Q_{3,3})=t^{20}-r_1t^{19}-r_2t^{18}-\cdots-r_{20}
\]
and put $\varpi=a_{-3}a_3$.  Theorem~\ref{thm:reciprocal} specializes to
\begin{equation}\label{eq:33-reciprocity}
 r_{20-i}=\varpi^{10-i}r_i\quad(1\le i\le9),
 \qquad r_{20}=-\varpi^{10}.
\end{equation}
The first three coefficients are
\begin{align*}
 r_1={}&a_0,\\
 r_2={}&-a_{-1}a_1+a_{-2}a_2-a_{-3}a_3,\\
 r_3={}&a_2a_{-1}^2-a_{-2}a_3a_{-1}+a_{-2}a_1^2
       -2a_{-2}a_0a_2-a_{-3}a_1a_2+3a_{-3}a_0a_3.
\end{align*}
Thus the ten coefficients through $r_{10}$ determine the opposite half of
the degree-$20$ recurrence.  This is a concrete coefficient-level instance
of the complement symmetry behind the general reciprocity theorem, while the
large transfer matrix itself remains in the companion construction paper.
\end{example}

\subsection{Symmetric Toeplitz matrices}

Assume now that $a_{-j}=a_j$ for $1\le j\le m$.  Then $q$ is palindromic:
\[
 z^{2m}q(1/z)=q(z).
\]
Generically its roots can therefore be written
\[
 x_1,x_1^{-1},\ldots,x_m,x_m^{-1}.
\]
This reciprocal-root structure is central to the symmetric-polynomial
approach in \cite{AlexanderssonEtAl2021}.

For any finite ternary vector $\eps$, write
\begin{equation}\label{eq:ternary-support-size}
 s(\eps):=|\{i:\eps_i\ne0\}|.
\end{equation}

\begin{theorem}[Symmetric recurrence-degree bound]\label{thm:symmetric}
For symmetric $(m,m)$-banded Toeplitz matrices, the determinant sequence has
an annihilating recurrence of degree at most
\begin{equation}\label{eq:symorder}
 d_{\mathrm{sym}}=\frac{3^m+1}{2}.
\end{equation}
\end{theorem}

\begin{proof}
Consider a root product $\prod_{j\in J}z_j$ with $|J|=m$.  For each
reciprocal pair $\{x_i,x_i^{-1}\}$ there are four possibilities: select
$x_i$ only, select $x_i^{-1}$ only, select both, or select neither.  Its
contribution to the product is therefore $x_i^{\eps_i}$ with
$\eps_i\in\{-1,0,1\}$.

Let $s=s(\eps)$ for $\eps=(\eps_1,\ldots,\eps_m)$.  The $s$ single
selections use $s$ roots; each
zero coordinate contributes either zero or two roots.  Hence an $m$-element
subset can yield $\eps$ only when $s\equiv m\pmod2$.  Consequently every
Widom characteristic factor belongs to the set
\[
 (-1)^ma_m x_1^{\eps_1}\cdots x_m^{\eps_m},
 \qquad
 \eps_i\in\{-1,0,1\},
 \quad s(\eps)\equiv m\pmod2.
\]
The number of such exponent vectors is
\[
 \sum_{\substack{0\le s\le m\\s\equiv m\pmod2}}
 \binom ms2^s
 =\frac{3^m+1}{2}.
\]
Taking the product of $t-W$ over the distinct possible characteristic
values gives an annihilator of degree at most this number.  Specializations
can only lower the degree.
\end{proof}

For $m=2$, the bound is $5$, consistent with the explicit Chebyshev
factorizations and order reduction known for symmetric pentadiagonal Toeplitz
determinants; see Elouafi \cite{Elouafi2018}.  The theorem gives the general
root-product counting mechanism behind that low-bandwidth phenomenon.

\subsection{Symmetry-adapted row-column straightening}\label{subsec:symmetric-straightening}

The spectral argument above gives the sharp recurrence bound, but determinant
signs also produce a purely boundary-minor reduction that has no permanent
analogue.  We formulate it in the subset labels of
Proposition~\ref{prop:state-recall}.  For two $j$-subsets of $[m]$, written in increasing order as
\[
 I=(i_k)_{k=1}^j,\qquad J=(j_k)_{k=1}^j,
\]
write $I\preceq J$ when $i_k\le j_k$ for every $1\le k\le j$.

\begin{theorem}[Catalan straightening of symmetric determinant states]
\label{thm:symmetric-catalan-straightening}
Assume $a_{-s}=a_s$ for $1\le s\le m$.  At exchange level $j$, every
determinant row-column boundary state is a parameter-independent linear
combination of states whose two deficit sets, after the common normalization
of Proposition~\ref{prop:state-recall}, form a comparable pair
$I\preceq J$.  Consequently a symmetry-adapted row-column construction may be
carried out with at most
\begin{equation}\label{eq:symmetric-narayana}
 N_{m,j}
 =\frac{1}{m+1}\binom{m+1}{j}\binom{m+1}{j+1}
\end{equation}
states at level $j$, and with at most
\begin{equation}\label{eq:symmetric-catalan}
 \sum_{j=0}^m N_{m,j}
 =C_{m+1}
 =\frac{1}{m+2}\binom{2m+2}{m+1}
\end{equation}
states in total.  Here $N_{m,j}=N(m+1,j+1)$ is a Narayana number.
\end{theorem}

\begin{proof}
Undo the harmless translations used in the normalization of a level-$j$
row-column signature.  For a sufficiently large symmetric Toeplitz section
$T_N$, the resulting boundary state is, up to its fixed cofactor sign, a
complementary minor
\[
 M_{I,J}=\det T_N[I^c,J^c],
\]
where the deleted row and column positions, expressed in their common ordered
boundary window, are two $j$-subsets $I,J\subseteq[m]$.  On the Zariski-open
set where $T_N$ is invertible, Jacobi's complementary-minor identity gives
\begin{equation}\label{eq:symmetric-jacobi-boundary}
 M_{I,J}
 =(-1)^{\sum I+\sum J}\det(T_N)
   \det\!\bigl(T_N^{-1}[J,I]\bigr).
\end{equation}
The inverse is symmetric.  Hence all level-$j$ states, after removal of the
common factor $\det(T_N)$ and the cofactor signs, are $j\times j$ minors of
one symmetric matrix.

For a generic symmetric $m\times m$ matrix, the minors indexed by comparable
pairs $I\preceq J$ are the \emph{doset minors}.  The classical doset
straightening law for symmetric matrices of De Concini, Eisenbud, and Procesi
\cite{DeConciniEisenbudProcesi1982} gives a basis of standard monomials in
these comparable minors.  In the homogeneous component spanned by the
$j$-minors themselves, this says that every $j$-minor is a scalar linear
combination of doset $j$-minors and that the doset minors are linearly
independent.  Applying these universal linear identities
to $T_N^{-1}$ in \eqref{eq:symmetric-jacobi-boundary}, and then restoring the
common factor and signs, straightens every boundary state to states with
$I\preceq J$.  Because the resulting identities are polynomial identities in
the entries of $T_N$ after clearing the common determinant, they extend from
the invertible locus to all symmetric specializations.

It remains to count the comparable pairs.  By the reflection principle,
\begin{align}
 \#\{(I,J): |I|=|J|=j,\ I\preceq J\}
 &=\binom mj^2-\binom m{j-1}\binom m{j+1}\notag\\
 &=\frac{1}{m+1}\binom{m+1}{j}\binom{m+1}{j+1},
 \label{eq:symmetric-doset-count}
\end{align}
which is the Narayana number in \eqref{eq:symmetric-narayana}.  Summing the
Narayana numbers over $j=0,\ldots,m$ gives the Catalan number
\eqref{eq:symmetric-catalan}.  Since the straightening coefficients are
independent of the Toeplitz parameters, one may straighten after each Laplace
step, obtaining a closed symmetry-adapted transfer on at most these states.
\end{proof}

\begin{example}[The first extra determinant relation: $m=4$]
\label{ex:symmetric-m4-straightening}
For $m=4$ and $j=2$, transposition alone leaves $21$ unordered pairs of
$2$-subsets, whereas \eqref{eq:symmetric-narayana} gives only $20$ doset
states.  The missing dimension is visible in the elementary identity valid
for every symmetric matrix $X$:
\begin{equation}\label{eq:symmetric-m4-plucker}
 \Delta_{12,34}(X)-\Delta_{13,24}(X)+\Delta_{14,23}(X)=0.
\end{equation}
Indeed $12\preceq34$ and $13\preceq24$, whereas $14\npreceq23$, so the last
minor straightens to the first two.  Across all exchange levels the counts are
\[
 1,\ 10,\ 20,\ 10,\ 1,
\]
and hence the $70$ unrestricted states reduce to $C_5=42$ symmetry-adapted
determinant states.
\end{example}

\subsection{Primitive symplectic compression}\label{subsec:symmetric-primitive}

The Catalan straightening has a natural exterior-power interpretation.  We
use it to recover the full power-of-three upper bound from the autonomous
transfer itself, without appealing to the Widom expansion of the scalar
sequence.  The companion paper identifies the balanced determinant evolution
with the compound transfer
\[
 T=(-1)^m a_m\,\bigwedge^m C_q,
\]
where $C_q$ is the Frobenius companion matrix of $q(z)/a_m$
\cite{AlekseyevKhomovskyRecurrences2026}.  We work over a splitting field on
the open set where the roots are simple.

Let $V$ be a $2m$-dimensional symplectic vector space, and choose a
symplectic basis $\{e_i,f_i\}_{i=1}^m$, so that
\[
 \Omega(e_i,f_j)=\delta_{ij},\qquad
 \Omega(e_i,e_j)=\Omega(f_i,f_j)=0.
\]
Write $\Lambda_\Omega:\bigwedge^mV\to\bigwedge^{m-2}V$ for contraction by
$\Omega$ and
\begin{equation}\label{eq:primitive-middle}
 P_m:=\ker\Lambda_\Omega.
\end{equation}
The standard symplectic Lefschetz decomposition gives
\begin{equation}\label{eq:primitive-dimension}
 \dim P_m
 =\binom{2m}{m}-\binom{2m}{m-2}
 =C_{m+1}.
\end{equation}
See, for example, \cite{FultonHarris1991} for the primitive exterior-power
realization of the fundamental symplectic representations.

The next theorem has a simple counting interpretation.  Each reciprocal
pair contributes exponent $+1$, $-1$, or $0$ to a transfer mode.  Several
boundary states can realize the same ternary exponent vector, and those copies
form a Catalan-number multiplicity space.  For one fixed autonomous transfer,
however, every copy of the same weight carries the same eigenvalue, so the
minimal polynomial counts distinct ternary weights rather than all Catalan
states.

\begin{theorem}[Primitive symplectic state and weight compression]
\label{thm:symmetric-primitive}
Assume $a_{-s}=a_s$ for $1\le s\le m$.  On the generic simple-root locus,
the $C_{m+1}$-dimensional symmetry-adapted determinant state space of
Theorem~\ref{thm:symmetric-catalan-straightening} is a realization of the
primitive middle exterior module $P_m$.  If
\[
 q(z)=a_m\prod_{i=1}^m(z-x_i)(z-x_i^{-1}),
\]
then the autonomous transfer on this module has weights
\begin{equation}\label{eq:primitive-weight-set}
 \eps=(\eps_1,\ldots,\eps_m)\in\{-1,0,1\}^m,
 \qquad s(\eps)\equiv m\pmod2,
\end{equation}
with transfer eigenvalue
\begin{equation}\label{eq:primitive-weight-eigenvalue}
 \lambda_\eps=(-1)^m a_m
 x_1^{\eps_1}\cdots x_m^{\eps_m}.
\end{equation}
If $s=s(\eps)$ and $m-s=2r$, then the multiplicity of this
weight in $P_m$ is the Catalan number
\begin{equation}\label{eq:primitive-weight-multiplicity}
 \operatorname{mult}_{P_m}(\eps)=C_r
 =\frac1{r+1}\binom{2r}{r}.
\end{equation}
Consequently the minimal polynomial of the symmetry-adapted autonomous
transfer has degree at most
\begin{equation}\label{eq:primitive-minpoly-bound}
 \#\{\eps\text{ in }\eqref{eq:primitive-weight-set}\}
 =\frac{3^m+1}{2},
\end{equation}
and this transfer minimal-polynomial degree is generically exactly
$(3^m+1)/2$.
\end{theorem}

\begin{proof}
We first identify the Catalan state space.  For a symmetric $m\times m$
matrix $X=(x_{ij})$, put
\[
 u_i=e_i+\sum_{j=1}^m x_{ji}f_j,
 \qquad
 p(X)=u_1\wedge\cdots\wedge u_m.
\]
Because $X=X^{\mathsf T}$,
\[
 \Omega(u_i,u_j)=x_{ij}-x_{ji}=0,
\]
so the graph of $X$ is Lagrangian and therefore
$\Lambda_\Omega p(X)=0$.  The Pl\"ucker coordinates of $p(X)$ are, up to the
usual signs, precisely all minors of $X$.  In the proof of
Theorem~\ref{thm:symmetric-catalan-straightening}, Jacobi's identity turns
the row-column boundary states, after removal of a common determinant factor,
into exactly such minors of a symmetric matrix.  Their doset straightening
has dimension $C_{m+1}$, equal to \eqref{eq:primitive-dimension}; hence the
resulting coordinate space is naturally dual to the full primitive module
$P_m$.  Using the standard symplectic self-duality of $P_m$, we identify the
two spaces.  Thus the Catalan reduction is a boundary-minor coordinate
realization of the primitive middle module.

Now diagonalize the reciprocal companion transfer over the splitting field.
Choose $e_i,f_i$ to be eigenvectors of $C_q$ with eigenvalues $x_i$ and
$x_i^{-1}$, respectively, and normalize the symplectic form so that the two
vectors in each reciprocal pair are paired.  Then $C_q$ preserves $\Omega$,
so contraction commutes with its exterior-power action and $P_m$ is
invariant.  The compound factor $(-1)^ma_m$ gives
\eqref{eq:primitive-weight-eigenvalue}.

A wedge basis vector has weight $\eps_i=1$ if it uses $e_i$ but not $f_i$,
weight $-1$ if it uses $f_i$ but not $e_i$, and weight $0$ if it uses both or
neither.  If $s$ coordinates are nonzero, the remaining $m-s=2r$ coordinates
must contribute exactly $r$ doubled reciprocal pairs in degree $m$.
Therefore the $\eps$-weight multiplicity in $\bigwedge^mV$ is
$\binom{2r}{r}$.  The same weight occurs in $\bigwedge^{m-2}V$ with
multiplicity $\binom{2r}{r-1}$.  Symplectic contraction is surjective in
middle degree and preserves weights, so
\[
 \dim(P_m)_\eps
 =\binom{2r}{r}-\binom{2r}{r-1}
 =C_r,
\]
proving \eqref{eq:primitive-weight-multiplicity}.  In particular every
admissible weight occurs.

For a fixed autonomous transfer, all $C_r$ copies of one weight carry the
same scalar eigenvalue \eqref{eq:primitive-weight-eigenvalue}; multiplicity
therefore contributes only one factor to the minimal polynomial.  Counting
the admissible ternary vectors gives
\[
 \sum_{\substack{0\le s\le m\\s\equiv m\ (2)}}\binom ms2^s
 =\frac{3^m+1}{2},
\]
which proves \eqref{eq:primitive-minpoly-bound}.  Finally take
$x_i=t^{3^{i-1}}$ with $t>1$.  Balanced-ternary uniqueness makes all the
numbers in \eqref{eq:primitive-weight-eigenvalue} distinct, so every
admissible weight contributes a distinct eigenvalue and the displayed degree
is attained.  Hence it is the generic minimal-polynomial degree of the
symmetry-adapted transfer.
\end{proof}

\begin{corollary}[Catalan multiplicities versus distinct autonomous modes]
\label{cor:symmetric-catalan-weight-identity}
The primitive weight decomposition gives the identities
\begin{align}
 C_{m+1}
 &=\sum_{\substack{0\le s\le m\\s\equiv m\ (2)}}
   \binom ms2^s C_{(m-s)/2},
 \label{eq:catalan-weight-sum}\\
 C_{m+1}-\frac{3^m+1}{2}
 &=\sum_{\substack{0\le s\le m\\s\equiv m\ (2)}}
   \binom ms2^s\bigl(C_{(m-s)/2}-1\bigr).
 \label{eq:catalan-weight-excess}
\end{align}
Thus the first reduction
$\binom{2m}{m}\to C_{m+1}$ removes universal determinant relations among
symmetric boundary minors, while the second
$C_{m+1}\to(3^m+1)/2$ forgets repeated copies of the same weight for the
fixed autonomous transfer.
\end{corollary}

\begin{remark}[The first two multiplicity drops]
\label{rem:symmetric-multiplicity-drops}
For $m=4$, the only repeated primitive weight is the zero weight.  It has
multiplicity $C_2=2$, so
\[
 70\longrightarrow42\longrightarrow41.
\]
Thus the final one-dimensional reduction is the extra copy of the zero
weight, whose eigenvalue is $a_4$.

For $m=5$, the ten weights with exactly one nonzero coordinate each have
multiplicity $C_2=2$, while all other weights are simple.  Hence
\[
 252\longrightarrow132\longrightarrow122.
\]
The characteristic polynomial of the ten redundant copies is
\[
 \prod_{i=1}^5(t+a_5x_i)(t+a_5x_i^{-1})
 =a_5^9 q\!\left(-\frac{t}{a_5}\right).
\]
These examples explain the first discrepancies between Catalan state count
and power-of-three recurrence degree without introducing additional ad hoc
minor identities.
\end{remark}

\begin{remark}[State count, transfer minimal polynomial, and scalar sharpness]
\label{rem:catalan-versus-widom}
The three relevant quantities should be kept distinct.  For example,
\[
\begin{array}{c|rrrr}
 m&\binom{2m}{m}&\dfrac{\binom{2m}{m}+2^m}{2}&C_{m+1}&\dfrac{3^m+1}{2}\\ \hline
 4&70&43&42&41\\
 5&252&142&132&122\\
 6&924&494&429&365
\end{array}
\]
The second column after the unrestricted count is the transposition-orbit
count that applies directly to symmetric permanents.  The Catalan column is
the determinant-specific boundary-state dimension.  The last column is the
generic minimal-polynomial degree of the autonomous primitive transfer, as
well as the Widom mode count.  Theorem~\ref{thm:symmetric-primitive} therefore
recovers the complete symmetric determinant upper bound from transfer
symmetry.  To conclude that the scalar principal determinant coordinate
actually sees every generic mode one still needs observability/noncancellation;
that is supplied independently by the Widom--Hankel argument in
Theorem~\ref{thm:sharp}.  The two viewpoints are complementary rather than
redundant.
\end{remark}

\subsection{Skew-symmetric Toeplitz matrices}

Assume $a_0=0$ and $a_{-j}=-a_j$.  Then
\[
 z^{2m}q(1/z)=-q(z),
\]
so $1$ and $-1$ are roots and the remaining roots occur in reciprocal
pairs.  Generically we may write them as
\[
 1,-1,x_1,x_1^{-1},\ldots,x_{m-1},x_{m-1}^{-1}.
\]

\begin{proposition}[Half-size square factorization for skew Toeplitz sections]
\label{prop:skew-toeplitz-square}
Let
\[
 A_{2r}=(c_{j-i})_{i,j=1}^{2r},
 \qquad c_0=0,\qquad c_{-s}=-c_s,
\]
be any even-order skew-symmetric Toeplitz matrix (with $c_s=0$ outside the
prescribed band, when the matrix is banded).  Define the $r\times r$
Toeplitz-minus-Hankel matrix
\begin{equation}\label{eq:skew-toeplitz-half-matrix}
 B_r=\bigl(c_{j-i}-c_{2r+1-i-j}\bigr)_{i,j=1}^r.
\end{equation}
Then
\begin{equation}\label{eq:skew-toeplitz-half-square}
 \det A_{2r}=\det(B_r)^2.
\end{equation}
With the standard Pfaffian convention,
\begin{equation}\label{eq:skew-toeplitz-pf-half}
 \operatorname{Pf}(A_{2r})=(-1)^r\det B_r.
\end{equation}
For odd order, $\det A_{2r+1}=0$.
\end{proposition}

\begin{proof}
The odd-order assertion follows from
$\det A=\det A^T=\det(-A)$.  For order $2r$, let $J$ be the reversal matrix.
Toeplitz skew-symmetry gives
\[
 J A_{2r}J=-A_{2r},
\]
so $A_{2r}$ interchanges the $+1$ and $-1$ eigenspaces of $J$.  In the
orthonormal bases
\[
 u_i=\frac{e_i+e_{2r+1-i}}{\sqrt2},\qquad
 v_i=\frac{e_i-e_{2r+1-i}}{\sqrt2},
 \qquad 1\le i\le r,
\]
a direct calculation gives
\[
 Q^TA_{2r}Q=
 \begin{pmatrix}
 0&B_r\\
 -B_r^T&0
 \end{pmatrix}.
\]
Taking determinants yields \eqref{eq:skew-toeplitz-half-square}.  The
classical identity $\det A=\operatorname{Pf}(A)^2$ for skew-symmetric
matrices goes back to Cayley; see, for example, Knuth's historical discussion
\cite{Knuth1996}.  Tracking the orientation of the displayed basis, or
applying the standard block Pfaffian formula, gives
\eqref{eq:skew-toeplitz-pf-half}.
\end{proof}

\begin{theorem}[Skew-symmetric recurrence-degree bound]\label{thm:skew}
For skew-symmetric $(m,m)$-banded Toeplitz matrices, the determinant sequence
has an annihilating recurrence of degree at most
\begin{equation}\label{eq:skeworder}
 d_{\mathrm{skew}}=2\cdot3^{m-1}.
\end{equation}
\end{theorem}

\begin{proof}
For each of the $m-1$ reciprocal pairs, the product of a selected subset
contributes an exponent in $\{-1,0,1\}$ exactly as above.  The two additional
roots $1$ and $-1$ can only contribute an overall sign.  Thus every Widom
characteristic factor belongs to
\[
 \left\{
 \pm(-1)^ma_m x_1^{\eps_1}\cdots x_{m-1}^{\eps_{m-1}}:
 \eps_i\in\{-1,0,1\}
 \right\},
\]
which has at most $2\cdot3^{m-1}$ elements.  Their product polynomial
annihilates the determinant sequence.
\end{proof}

\subsection{Skew fixed-ambient parity and Hodge two-step compression}
\label{subsec:skew-transfer-compression}

The skew case also admits a transfer explanation of the power-of-three bound,
but an important distinction from the symmetric case is needed.  If all
boundary minors are taken inside one fixed skew-symmetric ambient section,
transposition gives immediate sign relations and a half-dimensional coordinate
count.  The normalized row-column transfer, however, represents different
exchange levels using different ambient Toeplitz sections.  Therefore those
fixed-ambient identities do not define a levelwise quotient of the normalized
one-step row-column state vector.  The canonical half-dimensional reduction
instead lives in the compound realization, where it is the middle Hodge
splitting.

For the scalar sequence alone, the even-subsequence order bound can also be
read off from odd-order vanishing together with the even annihilator of the full
sequence.  The theorem below is retained for a different reason: it identifies
the transfer-level mechanism behind that halving, namely the Hodge decomposition
and the ternary spectrum of the two-step compound transfer.

\begin{theorem}[Skew fixed-ambient parity and Hodge two-step compression]
\label{thm:skew-transfer-compression}
Assume characteristic different from $2$, $a_0=0$, and
$a_{-s}=-a_s$ for $1\le s\le m$.

For one fixed ambient $n\times n$ skew-symmetric Toeplitz section, the
level-$j$ complementary boundary minors satisfy
\begin{equation}\label{eq:skew-boundary-transpose}
 M^{(n,j)}_{B,A}=(-1)^{n-j}M^{(n,j)}_{A,B}.
\end{equation}
Hence their fixed-ambient coordinate span has dimension at most
\begin{equation}\label{eq:skew-level-state-count}
 K^{(n)}_{m,j}
 =\frac12\left(\binom mj^2+(-1)^{n-j}\binom mj\right)
\end{equation}
at level $j$, and
\begin{equation}\label{eq:skew-half-state-count}
 \sum_{j=0}^m K^{(n)}_{m,j}
 =\frac12\binom{2m}{m}.
\end{equation}
This is a statement about complementary minors in a common ambient section;
it is not a half-dimensional quotient of the normalized one-step row-column
transfer.

On the generic simple-root locus, write
\[
 q(z)=a_m(z^2-1)\prod_{i=1}^{m-1}(z-x_i)(z-x_i^{-1}).
\]
For the compound realization
\[
 T=(-1)^ma_m\,\bigwedge^m C_q,
\]
the middle exterior space has the Hodge decomposition
\begin{equation}\label{eq:skew-hodge-halves}
 \bigwedge^mV=H_+\oplus H_-,
 \qquad
 \dim H_+=\dim H_-=\frac12\binom{2m}{m}.
\end{equation}
The one-step compound transfer exchanges $H_+$ and $H_-$, while $T^2$
preserves each half.  On either half the two-step transfer has eigenvalues
\begin{equation}\label{eq:skew-two-step-eigenvalue}
 \mu_\eps
 =a_m^2x_1^{2\eps_1}\cdots x_{m-1}^{2\eps_{m-1}},
 \qquad
 \eps\in\{-1,0,1\}^{m-1}.
\end{equation}
If $s=s(\eps)$ and $h=m-1-s$, then the multiplicity of
$\mu_\eps$ in either Hodge half is
\begin{equation}\label{eq:skew-two-step-multiplicity}
 \binom{h}{\lfloor h/2\rfloor}.
\end{equation}
Consequently the two-step compound transfer has minimal-polynomial degree at
most $3^{m-1}$ on either Hodge half, generically exactly $3^{m-1}$; the
one-step compound transfer is annihilated by a polynomial of degree at most
$2\cdot3^{m-1}$, generically with that minimal-polynomial degree.
\end{theorem}

\begin{proof}
For the fixed-ambient statement, write a level-$j$ complementary minor as
\[
 M^{(n,j)}_{A,B}=\det A_n[A^c,B^c],
\]
where the complementary matrix has order $n-j$.  Since
$A_n^{\mathsf T}=-A_n$,
\[
 A_n[B^c,A^c]
 =-\bigl(A_n[A^c,B^c]\bigr)^{\mathsf T},
\]
and taking determinants gives \eqref{eq:skew-boundary-transpose}.  If
$d_j=\binom mj$, the fixed-ambient level-$j$ matrix of complementary minors
is symmetric when $n-j$ is even and skew-symmetric when $n-j$ is odd.  Its
independent coordinates therefore number $d_j(d_j+1)/2$ or
$d_j(d_j-1)/2$, proving \eqref{eq:skew-level-state-count}.  Summing over $j$
and using
\[
 \sum_{j=0}^m\binom mj^2=\binom{2m}{m},
 \qquad
 \sum_{j=0}^m(-1)^j\binom mj=0
\]
proves \eqref{eq:skew-half-state-count}.

This count must not be confused with a direct coordinate quotient of the
normalized row-column transfer recalled in
Proposition~\ref{prop:state-recall}.  Normalization represents different
exchange levels by minors of different ambient Toeplitz sections; at one
transfer time the relevant minors therefore do not share the same ambient
size $n$.  The sign $(-1)^{n-j}$ in
\eqref{eq:skew-boundary-transpose} consequently cannot be imposed level by
level on the normalized state vector to obtain an invariant one-step
half-space.

For the autonomous reduction, use the compound realization from the companion
paper.  Over the splitting field choose eigenvectors of $C_q$ with
eigenvalues
\[
 1,-1,x_1,x_1^{-1},\ldots,x_{m-1},x_{m-1}^{-1}.
\]
These reciprocal pairs, together with the self-reciprocal eigenvalues $1$
and $-1$, admit a nondegenerate symmetric bilinear form preserved by $C_q$.
Since the product of all roots is $-1$, $C_q$ is orientation reversing for
this orthogonal form.  In middle exterior degree, the standard Hodge
splitting \eqref{eq:skew-hodge-halves} is therefore exchanged by $C_q$; see
\cite{FultonHarris1991} for the middle exterior decomposition of the even
orthogonal representation.  Hence $T^2$ preserves each Hodge half.

Fix $\eps\in\{-1,0,1\}^{m-1}$ and put
$s=s(\eps)$, $h=m-1-s$.  A middle-degree wedge giving this
squared weight selects exactly one member of each of the $s$ indicated
reciprocal pairs.  Among the remaining $h$ reciprocal pairs it selects both
members from some pairs and neither from the others; the roots $1$ and $-1$
complete the total degree.  If $h$ is even, exactly one of $1,-1$ is selected
and there are $2\binom{h}{h/2}$ such wedges.  If $h$ is odd, either both or
neither of $1,-1$ are selected, giving
\[
 \binom{h}{(h-1)/2}+\binom{h}{(h+1)/2}
 =2\binom{h}{\lfloor h/2\rfloor}
\]
wedges.  Thus the $\mu_\eps$-eigenspace of $T^2$ in the full middle exterior
power has dimension $2\binom{h}{\lfloor h/2\rfloor}$.  Because $T$ is
invertible and exchanges $H_+$ and $H_-$ while commuting with $T^2$, it
identifies the two $\mu_\eps$-eigenspaces.  Each Hodge half therefore has
multiplicity \eqref{eq:skew-two-step-multiplicity}.

The factor $(-1)^ma_m$ in $T$ disappears after squaring, and the special
roots $1,-1$ also square to $1$, which gives
\eqref{eq:skew-two-step-eigenvalue}.  There are $3^{m-1}$ such ternary
weights, so the two-step minimal polynomial has degree at most $3^{m-1}$.
Taking $x_i=t^{3^{i-1}}$ with $t>1$ makes the values
\eqref{eq:skew-two-step-eigenvalue} pairwise distinct by balanced-ternary
uniqueness, proving generic equality for the two-step compound transfer.  If
$R$ is its minimal polynomial, then $R(T^2)=0$, so $R(t^2)$ annihilates the
one-step compound transfer and has degree $2\cdot3^{m-1}$.  At the same
separated-root specialization the corresponding one-step modes occur in
opposite-sign pairs, giving generic equality for the one-step compound
minimal-polynomial degree as well.
\end{proof}

\begin{corollary}[Hodge multiplicities versus two-step modes]
\label{cor:skew-weight-identity}
For either Hodge half,
\begin{equation}\label{eq:skew-half-weight-sum}
 \frac12\binom{2m}{m}
 =\sum_{s=0}^{m-1}
   \binom{m-1}{s}2^s
   \binom{m-1-s}{\lfloor(m-1-s)/2\rfloor}.
\end{equation}
Thus the skew-symmetric compound reduction has the two stages
\[
 \binom{2m}{m}
 \longrightarrow \frac12\binom{2m}{m}
 \longrightarrow 3^{m-1}
\]
for the two-step system, while the full one-step system has
$2\cdot3^{m-1}$ generic autonomous modes.  The first arrow is the canonical
middle Hodge splitting; the second forgets multiplicities of equal squared
weights.  The fixed-ambient boundary-minor count
\eqref{eq:skew-half-state-count} has the same half-binomial size, but it is
not a direct quotient of the normalized row-column transfer.
\end{corollary}

\begin{remark}[First skew multiplicity drops]
\label{rem:skew-first-drops}
The first values are
\[
\begin{array}{c|ccc}
 m&\binom{2m}{m}&\frac12\binom{2m}{m}&3^{m-1}\\ \hline
 1&2&1&1\\
 2&6&3&3\\
 3&20&10&9\\
 4&70&35&27\\
 5&252&126&81
\end{array}
\]
so the first genuine two-step multiplicity collapse occurs at $m=3$.  The
middle column is the dimension of either Hodge half (and also the count of
independent fixed-ambient complementary minors).  The full skew compound
transfer correspondingly has autonomous degree $2\cdot3^{m-1}$, while the
even-size scalar subsequence is governed by the $3^{m-1}$ squared modes.
Generic scalar minimality of these bounds is proved independently below by
the Widom--Hankel argument.
\end{remark}

\subsection{Algorithmic realization}\label{subsec:symmetry-aware-algorithm}

The symmetric state-space reduction can be implemented directly as a wrapper
around the row-column construction recalled in
Proposition~\ref{prop:state-recall}: symmetric boundary minors are
straightened as soon as they are generated.  The skew case requires a
different implementation.  The fixed-ambient transpose identity
\eqref{eq:skew-boundary-transpose} is useful structurally, but its parity sign
is not compatible with a levelwise quotient of the normalized row-column
state vector.  For skew software one therefore keeps the full specialized
row-column transfer, squares it, and performs the Krylov/minimal-polynomial
calculation on the resulting two-step evolution.

\begin{quote}
\small
\textbf{Pseudocode:} \textsc{SymmetryAwareRowColumn}$(m,\mathrm{type})$.

\textbf{Input:} semibandwidth $m$ and a balanced Toeplitz symbol satisfying
$\mathrm{type}\in\{\mathrm{symmetric},\mathrm{skew}\}$.

\textbf{Output:} an annihilating polynomial for the determinant sequence and,
when requested, the corresponding reduced autonomous transfer information.

\begin{enumerate}[leftmargin=*,label=\textbf{\arabic*.}]
\item Generate the reachable normalized row-column signatures
      $\Sigma(A,B)$ level by level, exactly as in the unrestricted
      row-column construction, and assemble the specialized determinant
      transfer $Q$.

\item If $\mathrm{type}=\mathrm{symmetric}$, then whenever a level-$j$
      boundary minor is generated, straighten it immediately to the doset
      basis $I\preceq J$ from
      Theorem~\ref{thm:symmetric-catalan-straightening}.  Assemble the
      resulting transfer $S_m$ on the $C_{m+1}$ Catalan states.

\item Compute the minimal polynomial $M_S(t)$ of $S_m$ (or obtain it by a
      Krylov calculation).  Every determinant coordinate is annihilated by
      $M_S$.  For generic symmetric parameters,
      \[
        \deg M_S=\frac{3^m+1}{2}.
      \]

\item If $\mathrm{type}=\mathrm{skew}$, specialize the full normalized
      row-column transfer to $a_0=0$ and $a_{-s}=-a_s$.  Do \emph{not}
      quotient its normalized signatures level by level using
      \eqref{eq:skew-boundary-transpose}.  Instead form the two-step transfer
      \[
        U=Q^2.
      \]

\item For the scalar determinant recurrence, apply a principal-coordinate
      Krylov calculation directly to $U$.  Let $R_{\rm sc}(y)$ be the
      resulting annihilating polynomial for the even subsequence
      $E_k=D_{2k}$.  Since $D_{2k+1}=0$,
      \[
        R_{\rm sc}(t^2)
      \]
      annihilates the full sequence.  Generically,
      \[
        \deg R_{\rm sc}=3^{m-1},\qquad
        \deg R_{\rm sc}(t^2)=2\cdot3^{m-1}.
      \]
      If a full transfer annihilator is desired instead, compute the minimal
      polynomial of $U$ and substitute $y=t^2$; this always gives a valid
      one-step transfer annihilator, though the scalar Krylov route is the
      quantity needed for the determinant recurrence.

\item For a special parameter specialization, return the actually computed
      polynomial; its degree may be smaller than the generic value.
\end{enumerate}
\end{quote}

Implementations accompanying these algorithms are maintained at
\url{https://github.com/maxale/matrix_codes}.
The same repository also contains the implementations accompanying the
companion recurrence paper \cite{AlekseyevKhomovskyRecurrences2026}.

\subsection{Generic sharpness}

The bounds above are in fact the generic minimal degrees.  The following
standard factorization isolates the only possible obstruction.

\begin{lemma}[Hankel--Vandermonde criterion]\label{lem:hankel-vandermonde}
Let
\[
 u_n=\sum_{\nu=1}^N A_\nu\lambda_\nu^n
\]
with pairwise distinct $\lambda_1,\ldots,\lambda_N$.  Then
\begin{equation}\label{eq:hankel-vandermonde}
 \det\bigl(u_{i+j}\bigr)_{0\le i,j<N}
 =\left(\prod_{\nu=1}^N A_\nu\right)
  \prod_{1\le\mu<\nu\le N}(\lambda_\nu-\lambda_\mu)^2.
\end{equation}
In particular, if every $A_\nu$ is nonzero, then the minimal homogeneous
constant-coefficient recurrence for $(u_n)$ has degree $N$.
\end{lemma}

\begin{proof}
Let $V=(\lambda_\nu^i)_{0\le i<N,\,1\le\nu\le N}$.  The Hankel matrix
factors as
\[
 (u_{i+j})_{0\le i,j<N}
 =V\,\operatorname{diag}(A_1,\ldots,A_N)V^{\mathsf T}.
\]
Taking determinants and using the Vandermonde determinant gives
\eqref{eq:hankel-vandermonde}.  If the determinant is nonzero, the Hankel
rank is at least $N$, so no recurrence of degree less than $N$ can annihilate
the sequence; the displayed exponential representation gives one of degree
$N$.
\end{proof}

\begin{theorem}[Generic sharpness]\label{thm:sharp}
Within each of the symmetric and skew-symmetric parameter spaces, let $H$
denote the Hankel determinant of the predicted size from
Lemma~\ref{lem:hankel-vandermonde}.  Then $H$ is a nonzero polynomial in the
diagonal parameters, and at every parameter point with $H\ne0$ the minimal
recurrence degrees in Theorems~\ref{thm:symmetric} and \ref{thm:skew} are
respectively
\[
 \frac{3^m+1}{2}
 \qquad\text{and}\qquad
 2\cdot3^{m-1}.
\]
Equivalently, these degrees are generic: a drop can occur only on the proper
algebraic set $H=0$.
\end{theorem}

\begin{proof}
For simple roots, Widom's coefficient in the balanced case may be written
\begin{equation}\label{eq:widom-coefficient-balanced}
 C_J=
 \frac{\displaystyle\prod_{j\in J}z_j^m}
 {\displaystyle\prod_{\substack{j\in J\\k\notin J}}(z_j-z_k)},
 \qquad |J|=m,
\end{equation}
see \cite{Widom1958,Alexandersson2012}.  If several subsets give the same
value of $W_J$, group their terms in \eqref{eq:widomform}; the coefficient of
a distinct mode $\lambda$ is
\[
 A_\lambda\defeq\sum_{J:W_J=\lambda}C_J.
\]
For either symmetry class the finite Hankel determinant of the corresponding
predicted size is a polynomial in the diagonal parameters, because every
$D_n$ is.  It therefore suffices to exhibit one specialization for which the
predicted number of modes are distinct and all the grouped coefficients are
nonzero.

\emph{Symmetric case.}
Fix a real number $t>1$ and take
\begin{equation}\label{eq:separated-symmetric}
 x_i=t^{3^{i-1}},\qquad
 q(z)=\prod_{i=1}^m(z-x_i)(z-x_i^{-1}).
\end{equation}
This monic polynomial is palindromic, so it defines a symmetric Toeplitz
symbol with $a_m=a_{-m}=1$.  Every admissible exponent vector
$\eps\in\{-1,0,1\}^m$ from Theorem~\ref{thm:symmetric} is realized: if $s$
coordinates are nonzero, choose both roots from exactly $(m-s)/2$ of the
zero coordinates.  Its grouped mode is
\[
 \lambda_\eps=(-1)^m t^{\sum_{i=1}^m\eps_i3^{i-1}}.
\]
These modes are pairwise distinct by uniqueness of balanced ternary
expansion.

It remains to rule out cancellation inside a group.  Order the roots as
\[
 x_m^{-1}<\cdots<x_1^{-1}<x_1<\cdots<x_m.
\]
If $J=\{j_1<\cdots<j_m\}$ denotes the positions of the selected roots in
this ordering, then
\begin{equation}\label{eq:cross-sign}
 \operatorname{sgn}
 \prod_{\substack{j\in J\\k\notin J}}(z_j-z_k)
 =(-1)^{\sum_{a=1}^m(m+a-j_a)}.
\end{equation}
Indeed, for the root in position $j_a$, exactly $m+a-j_a$ unselected roots
lie to its right.  The two positions occupied by $x_i^{-1}$ and $x_i$ sum
to $2m+1$.  Within a fixed $\eps$-group the singly selected roots are fixed,
and the number $(m-s)/2$ of doubly selected reciprocal pairs is fixed.
Hence the parity of $\sum_a j_a$, and therefore the sign in
\eqref{eq:cross-sign}, is constant throughout the group.  The numerator in
\eqref{eq:widom-coefficient-balanced} is positive.  Thus all $C_J$ in a
fixed group have the same nonzero sign, so every $A_{\lambda_\eps}$ is
nonzero.  Lemma~\ref{lem:hankel-vandermonde} now gives a nonzero Hankel
determinant of size $(3^m+1)/2$ at this specialization.

\emph{Skew-symmetric case.}
Use instead
\begin{equation}\label{eq:separated-skew}
 q(z)=(z^2-1)\prod_{i=1}^{m-1}(z-x_i)(z-x_i^{-1}),
 \qquad x_i=t^{3^{i-1}}.
\end{equation}
Then $z^{2m}q(1/z)=-q(z)$, so the associated Toeplitz symbol is
skew-symmetric.  Its candidate modes are
\[
 \lambda_{\sigma,\eps}
 =(-1)^m\sigma t^{\sum_{i=1}^{m-1}\eps_i3^{i-1}},
 \qquad
 \sigma\in\{\pm1\},\quad
 \eps\in\{-1,0,1\}^{m-1},
\]
and all $2\cdot3^{m-1}$ of them are distinct.

Every such mode is realized.  Put
$s=|\{i:\eps_i\ne0\}|$ and $h=m-s$.  If $h$ is odd, select exactly one of
$1,-1$, choosing $1$ for $\sigma=1$ and $-1$ for $\sigma=-1$, and fill the
remaining $(h-1)/2$ pairs by double selections.  If $h$ is even, select
neither of $1,-1$ for $\sigma=1$ and both for $\sigma=-1$; the remaining
roots come from respectively $h/2$ or $(h-2)/2$ doubled zero pairs.  The
latter number is nonnegative because $s\le m-1$, so even $h$ is at least
$2$.

Order the roots as
\[
 -1<x_{m-1}^{-1}<\cdots<x_1^{-1}<1<x_1<\cdots<x_{m-1}.
\]
Now the two positions of every reciprocal pair sum to $2m+2$, an even
number.  For a fixed $(\sigma,\eps)$ the choices of $1$ and $-1$, the singly
selected reciprocal roots, and the number of doubled pairs are all fixed.
Consequently the parity in \eqref{eq:cross-sign} is again constant throughout
the group.  The sign of the numerator in
\eqref{eq:widom-coefficient-balanced} depends only on whether $-1$ is
selected and is therefore fixed as well.  Hence all $C_J$ in each group
again have the same nonzero sign.  Lemma~\ref{lem:hankel-vandermonde} gives
a nonzero Hankel determinant of size $2\cdot3^{m-1}$.

In each symmetry class the relevant Hankel determinant is therefore a
nonzero polynomial in the diagonal parameters.  Its nonvanishing defines a
nonempty Zariski-open set.  On that set Lemma~\ref{lem:hankel-vandermonde}
gives the corresponding lower bound on the minimal recurrence degree, while
Theorems~\ref{thm:symmetric} and \ref{thm:skew} give the matching upper
bound.
\end{proof}

\begin{corollary}[Even-size skew-symmetric subsequence]\label{cor:skew-even}
Assume the coefficient field has characteristic different from $2$.  For a
skew-symmetric $(m,m)$-banded Toeplitz matrix,
\[
 D_{2k+1}=0.
\]
If $E_k\defeq D_{2k}$, then $(E_k)$ has a homogeneous constant-coefficient
annihilator of degree at most
\begin{equation}\label{eq:skew-even-order}
 3^{m-1}.
\end{equation}
Moreover, this degree is minimal on a nonempty Zariski-open subset of the
skew-symmetric parameter space.
\end{corollary}

\begin{proof}
Since $A_n^{\mathsf T}=-A_n$,
\[
 D_n=\det A_n^{\mathsf T}=\det(-A_n)=(-1)^nD_n,
\]
so $D_n=0$ for odd $n$.  The same transpose/sign argument also shows that
the permanent of an odd-order skew-symmetric matrix vanishes.

In the proof of Theorem~\ref{thm:skew}, put
\[
 \Lambda_{\eps}=(-1)^ma_m
 x_1^{\eps_1}\cdots x_{m-1}^{\eps_{m-1}},
 \qquad \eps\in\{-1,0,1\}^{m-1}.
\]
The full determinant sequence is annihilated by the even polynomial
\[
 P(t)=\prod_{\eps\in\{-1,0,1\}^{m-1}}
       (t^2-\Lambda_{\eps}^{\,2})
     =R(t^2),
\]
where $\deg R=3^{m-1}$.  Restricting $P$ to even indices turns a two-step
shift of $(D_n)$ into a one-step shift of $(E_k)$, so $R$ annihilates the
even subsequence.

For generic sharpness, suppose that at a parameter point where
Theorem~\ref{thm:sharp} gives full minimal degree $2\cdot3^{m-1}$, the even
subsequence had an annihilator $Q$ of degree $r<3^{m-1}$.  Then
$Q(t^2)$ would annihilate the full sequence: it annihilates the even indices
by assumption, and every term occurring at an odd index is zero.  This would
give a full-sequence annihilator of degree $2r<2\cdot3^{m-1}$, a
contradiction.
\end{proof}

\begin{remark}\label{rem:exceptional-locus}
The locus on which the minimal recurrence degree drops should not be
identified with the discriminant locus of $q$.  Even when the roots of $q$
are distinct, special multiplicative relations can make two Widom
characteristic values coincide, and grouped Widom coefficients can also
vanish.  Conversely, if $q$ has repeated roots, then some of the formal
Widom characteristic values necessarily coincide, so the product over the
Widom factors acquires repeated factors.  Such a collision does not by itself
determine the minimal recurrence degree.  On the discriminant locus the
simple-root Widom expansion must instead be interpreted through its confluent
limit: colliding exponential modes may produce polynomial--exponential terms
of the form
$n^r W^n$, and such a term requires a factor $(t-W)^{r+1}$ in a minimal
annihilator.  Thus fewer distinct Widom values need not mean a lower
recurrence degree; part or all of the multiplicity can survive through
confluent terms.

The Hankel condition packages these mechanisms into an exact generic
certificate in the symmetry classes considered above: nonvanishing of the
Hankel determinant of order $d_{\mathrm{sym}}$ or $d_{\mathrm{skew}}$
forces the predicted minimal degree.  Thus every degree drop lies on the
corresponding hypersurface $H=0$; without additional information, however,
we do not identify every point of $H=0$ with a degree drop.

For the first nontrivial symmetric case $m=2$, the exceptional hypersurface
is particularly transparent.  A direct calculation from
$D_0,\ldots,D_8$ gives
\[
 H_5\defeq\det(D_{i+j})_{0\le i,j<5}
 =2a_1^2a_2^{18}.
\]
Hence $H_5=0$ is the union of the coordinate hyperplanes $a_1=0$ and
$a_2=0$; on the genuine pentadiagonal locus $a_2\ne0$ it reduces simply to
$a_1=0$.  This small-band example shows that explicit factorization and a
simple geometric description can occur, while no comparable factorization
is asserted here for general $m$.
\end{remark}

\section{Permanent--determinant conversion in small semibandwidth}\label{sec:conversion}

The general transfer construction treats determinants and permanents in parallel but does not normally identify them.  The zero-diagonal pentadiagonal support is exceptional: a fixed P\'olya--Kasteleyn signing converts every permanent to a determinant, even before Toeplitz constancy is imposed.  This sits within the classical theory of convertible matrices and P\'olya signings; see, in particular, Gibson's characterization and extremal results \cite{Gibson1969,Gibson1971} and the determinant--permanent relations studied by Tarakanov and Zatorskii \cite{TarakanovZatorskii2009}.  We first record the permanent recurrence from the six-state calculation in the companion paper \cite{AlekseyevKhomovskyRecurrences2026}:
\begin{align}
P_{n+6}={}&a_0P_{n+5}+(a_{-1}a_1+a_{-2}a_2)P_{n+4}
 +(a_{-2}a_1^2+a_{-1}^2a_2)P_{n+3}\notag\\
&+a_{-2}a_2(a_{-1}a_1+a_{-2}a_2)P_{n+2}
 -a_0(a_{-2}a_2)^2P_{n+1}-(a_{-2}a_2)^3P_n.
\label{eq:pentaperm}
\end{align}
\begin{proposition}[P\'olya conversion for the zero-diagonal pentadiagonal support]
\label{prop:penta-polya}
Let $B=(b_{ij})_{i,j=1}^n$ be any square matrix satisfying
\[
 b_{ij}=0\qquad\text{unless}\qquad 0<|i-j|\le2.
\]
Let $C=(c_{ij})$ be the sign matrix whose only relevant negative entries are
on the first superdiagonal,
\[
 c_{ij}=\begin{cases}
 -1,&j=i+1,\\
 1,&\text{otherwise}.
 \end{cases}
\]
Then
\begin{equation}\label{eq:penta-polya-general}
 \perm B=\det(C\circ B),
\end{equation}
where $\circ$ denotes the Hadamard product.  Thus the full coordinate
subspace with zero main diagonal and lower and upper semibandwidths two is
convertible by a signing rule independent of the matrix entries.
\end{proposition}

\begin{proof}
View the support as a bipartite graph $G_n$ with row vertices $r_i$, column
vertices $c_j$, and an edge $r_ic_j$ whenever $0<|i-j|\le2$.  For $n\ge4$
there is a plane embedding whose quadrilateral face boundaries are
\[
 r_i c_{i+1} r_{i+3} c_{i+2} r_i,
 \qquad
 c_i r_{i+1} c_{i+3} r_{i+2} c_i,
 \qquad 1\le i\le n-3,
\]
and whose two remaining face boundaries are the hexagons
\[
 r_1c_3r_2c_1r_3c_2r_1,
 \qquad
 r_{n-2}c_{n-1}r_nc_{n-2}r_{n-1}c_nr_{n-2}.
\]
These $2n-4$ cycles give the complete face set: each support edge occurs on
two listed boundaries and $|E(G_n)|-|V(G_n)|+2=(4n-6)-2n+2=2n-4$.
(The case $n=3$ is the corresponding $6$-cycle, and $n\le2$ is immediate.)
Give an edge weight $-1$ precisely when it is a first-superdiagonal edge
$r_ic_{i+1}$ and weight $+1$ otherwise.  Every quadrilateral face contains
one negative edge, whereas each hexagonal face contains two.  Hence the
product of the signing weights around every face $F$ is
\[
 (-1)^{|F|/2+1},
\]
which is exactly the flat Kasteleyn condition in the bipartite
permanent--determinant method; see, for example,
\cite{Kasteleyn1967,Kuperberg1998}.  Consequently all perfect-matching terms
in $\det(C\circ B)$ have the same sign.

It remains only to fix the global sign.  If $n$ is even, use the matching
corresponding to $(1\,2)(3\,4)\cdots(n-1\,n)$; its permutation sign is
$(-1)^{n/2}$ and it uses $n/2$ first-superdiagonal edges, so the two sign
factors cancel.  If $n\ge3$ is odd, use
$(1\,2\,3)(4\,5)\cdots(n-1\,n)$; its permutation sign is
$(-1)^{(n-3)/2}$, while the signing contributes
$(-1)^{2+(n-3)/2}$, again giving $+1$.
Thus the common determinant coefficient is $+1$, proving
\eqref{eq:penta-polya-general} for arbitrary edge weights.
\end{proof}

For Toeplitz matrices the converter has two particularly simple forms.
Write
\[
 A_n(a)=T_n(a_{-2},a_{-1},0,a_1,a_2),\qquad
 D_n(a)\defeq\det A_n(a),\qquad P_n(a)\defeq\perm A_n(a),
\]
and define
\[
 \Psi(a_{-2},a_{-1},0,a_1,a_2)
 :=(a_{-2},a_{-1},0,-a_1,a_2).
\]

\begin{corollary}[Real Toeplitz sign conversion]\label{cor:penta-real-conversion}
For every $n\ge0$,
\begin{equation}\label{eq:penta-real-conversion}
 P_n(a)=D_n(\Psi a),
 \qquad
 D_n(a)=P_n(\Psi a).
\end{equation}
\end{corollary}

\begin{proof}
The first identity is Proposition~\ref{prop:penta-polya}, since $\Psi$ changes
exactly the first superdiagonal.  Apply the same identity to $\Psi a$ and
use $\Psi^2=1$ to obtain the second.
\end{proof}

The experimentally natural form treats the two directions symmetrically.
Define the phase map
\begin{equation}\label{eq:penta-phase-map}
 \Phi(a_{-2},a_{-1},0,a_1,a_2)
 :=(-a_{-2},i a_{-1},0,i a_1,-a_2),
\end{equation}
that is, $a_s\mapsto i^{|s|}a_s$ for $s=\pm1,\pm2$.

\begin{corollary}[Toeplitz phase conversion]\label{cor:penta-phase-conversion}
For every $n\ge0$,
\begin{equation}\label{eq:penta-phase-conversion}
 P_n(a)=D_n(\Phi a),
 \qquad
 D_n(a)=P_n(\Phi a).
\end{equation}
\end{corollary}

\begin{proof}
Let $S_n=\operatorname{diag}(i,i^2,\ldots,i^n)$.  Since a Toeplitz entry on
offset $s=j-i$ is multiplied under $S_n^{-1}(\cdot)S_n$ by $i^s$,
\[
 S_n^{-1}A_n(\Phi a)S_n=A_n(\Psi a).
\]
The first identity follows from Corollary~\ref{cor:penta-real-conversion}.
Moreover $\Phi^2$ multiplies the odd offsets by $-1$ and leaves the even
offsets fixed.  This is the diagonal similarity induced by
$\operatorname{diag}((-1)^1,\ldots,(-1)^n)$, so
$D_n(\Phi^2a)=D_n(a)$.  Applying the first identity to $\Phi a$ yields the
second.
\end{proof}

\subsection{Restoring the main diagonal: paired fixed-point renewal}\label{subsec:penta-fixed-point}

The P\'olya conversion requires a zero main diagonal.  Once a constant main
diagonal is restored, direct permanent--determinant conversion no longer
persists, but the two sequences satisfy parallel renewal identities.  The
only difference is the sign of the unique transposition that can cross a
fixed point.  We first record the universal fixed-point expansion behind
this refinement.

\begin{lemma}[Universal fixed-point expansions]\label{lem:fixedpoint-universal}
For every $n\times n$ matrix $B$ and every $0\le k\le n$,
\begin{align}
 [c^k]\perm(B+cI_n)
 &=\sum_{\substack{S\subseteq[n]\\|S|=k}}
 \perm B[[n]\setminus S,[n]\setminus S],
 \label{eq:fixedpoint-universal-per}\\
 [c^k]\det(B+cI_n)
 &=\sum_{\substack{S\subseteq[n]\\|S|=k}}
 \det B[[n]\setminus S,[n]\setminus S].
 \label{eq:fixedpoint-universal-det}
\end{align}
\end{lemma}

\begin{proof}
In either expansion choose the summand $c$ from the diagonal exactly at the
positions in $S$.  Those positions are fixed, while the remaining
permutation is an arbitrary bijection of $[n]\setminus S$.  In the
determinant case adjoining fixed points does not change the permutation sign,
so the remaining signed sum is the corresponding principal minor.
\end{proof}

Fix the four off-diagonal parameters and write
\begin{align}
 \mathcal P_n(c)&=\perm T_n(a_{-2},a_{-1},c,a_1,a_2),
 & Z_n&=\mathcal P_n(0),\notag\\
 \mathcal D_n(c)&=\det T_n(a_{-2},a_{-1},c,a_1,a_2),
 & Y_n&=\mathcal D_n(0),
 \label{eq:penta-fixedpoint-defs}
\end{align}
and put
\begin{equation}\label{eq:penta-fixedpoint-invariants}
 p=a_{-1}a_1,\qquad t=a_{-2}a_2,\qquad
 v=a_{-2}a_1^2+a_{-1}^2a_2.
\end{equation}
For $\eps\in\{+1,-1\}$ define
\begin{equation}\label{eq:penta-fixedpoint-Qeps}
 Q_{\eps}(z)
 =1-(\eps p+t)z^2-vz^3-t(\eps p+t)z^4+t^3z^6,
\end{equation}
and use the unified notation
\begin{equation}\label{eq:penta-fixedpoint-unified-notation}
 F_n^{(\eps)}(c)=
 \begin{cases}
  \mathcal P_n(c),&\eps=+1,\\
  \mathcal D_n(c),&\eps=-1,
 \end{cases}
 \qquad
 Z_n^{(\eps)}=F_n^{(\eps)}(0),
 \qquad
 Z^{(\eps)}(z)=\sum_{n\ge0}Z_n^{(\eps)}z^n.
\end{equation}
Thus $Z^{(+1)}(z)=Z(z)$ and $Z^{(-1)}(z)=Y(z)$.

\begin{proposition}[Paired fixed-point renewal identities]
\label{prop:penta-fixedpoint-renewal}
For $\eps\in\{+1,-1\}$,
\begin{equation}\label{eq:penta-zero-gf-unified}
 Z^{(\eps)}(z)=\frac{1-tz^2}{Q_{\eps}(z)},
\end{equation}
and
\begin{align}
 \mathcal F_{\eps}(z,c)
 &:=\sum_{n\ge0}F_n^{(\eps)}(c)z^n\notag\\
 &=\frac{Z^{(\eps)}(z)}
 {1-cz(1+\eps tz^2)Z^{(\eps)}(z)}
 \label{eq:penta-renewal-gf}\\[3pt]
 &=\frac{1-tz^2}
 {Q_{\eps}(z)-cz(1+\eps tz^2)(1-tz^2)}.
 \label{eq:penta-full-gf-unified}
\end{align}
Equivalently, the two cases are
\begin{align}
 \mathcal P(z,c)
 &=\frac{Z(z)}{1-cz(1+tz^2)Z(z)}
 =\frac{1-tz^2}{Q_{+1}(z)-cz(1-t^2z^4)},
 \label{eq:penta-renewal-per}\\
 \mathcal D(z,c)
 &=\frac{Y(z)}{1-cz(1-tz^2)Y(z)}
 =\frac{1-tz^2}{Q_{-1}(z)-cz(1-tz^2)^2}.
 \label{eq:penta-renewal-det}
\end{align}
Subtracting $Z^{(\eps)}(z)$ from \eqref{eq:penta-renewal-gf} gives the
corresponding difference generating function immediately.  Moreover,
\begin{equation}\label{eq:penta-invariant-ring}
 F_n^{(\eps)}(c)\in\mathbb Z[c,p,t,v]
 \qquad(n\ge0,\ \eps=\pm1).
\end{equation}
\end{proposition}

\begin{proof}
For $\eps=+1$, recurrence~\eqref{eq:pentaperm} at $c=0$ gives
\[
 Z(z)=\frac{1-tz^2}{Q_{+1}(z)}.
\]
For $\eps=-1$, Corollary~\ref{cor:penta-real-conversion} gives
$Y_n=P_n(\Psi a)$ on the zero-diagonal locus.  Under $\Psi$ the invariant
$p$ changes to $-p$, while $t$ and $v$ remain fixed.  Replacing $p$ by
$-p$ in the preceding formula therefore gives
\[
 Y(z)=\frac{1-tz^2}{Q_{-1}(z)}.
\]

It remains to restore the main diagonal.  In a bandwidth-two permutation, a
fixed point $i$ either separates the permutation into independent left and
right pieces, contributing $cz$, or it is crossed.  In the latter case
bandwidth two forces $i-1\mapsto i+1$, and bijectivity then forces
$i+1\mapsto i-1$.  Thus the unique crossed configuration is
$(i-1\ i+1)(i)$, of size three and unsigned weight $tc$.  It contributes
$+tc$ to the permanent and $-tc$ to the determinant.  Hence in the unified
notation its contribution is $\eps tc$, and scanning from left to right gives
the unique decomposition
\[
 Z^{(\eps)}A_{\eps}Z^{(\eps)}A_{\eps}\cdots
 A_{\eps}Z^{(\eps)},
 \qquad
 A_{\eps}=cz(1+\eps tz^2).
\]
Summing the resulting geometric series proves
\eqref{eq:penta-renewal-gf}; substituting
\eqref{eq:penta-zero-gf-unified} gives
\eqref{eq:penta-full-gf-unified} and hence
\eqref{eq:penta-renewal-per}--\eqref{eq:penta-renewal-det}.

Finally, diagonal similarity shows that every Toeplitz monomial is balanced
with respect to displacement.  The balanced invariant ring is generated by
\[
 p,\quad t,\quad x=a_{-2}a_1^2,\quad y=a_{-1}^2a_2,
 \qquad xy=tp^2.
\]
Transposition fixes $p,t$ and interchanges $x,y$.  Both determinant and
permanent are invariant under transposition, so
\[
 \mathbb Z[p,t,x,y]^{x\leftrightarrow y}/(xy-tp^2)
 =\mathbb Z[p,t,x+y]=\mathbb Z[p,t,v],
\]
which proves \eqref{eq:penta-invariant-ring}.
\end{proof}

\begin{corollary}[All fixed-point layers]\label{cor:penta-fixedpoint-layers}
For every $k,n\ge0$ and $\eps\in\{+1,-1\}$,
\begin{equation}\label{eq:penta-fixedpoint-coeff}
 [c^k]F_n^{(\eps)}(c)
 =[z^{n-k}](1+\eps tz^2)^kZ^{(\eps)}(z)^{k+1}.
\end{equation}
Equivalently,
\begin{equation}\label{eq:penta-fixedpoint-convolution}
 [c^k]F_n^{(\eps)}(c)
 =\sum_{j=0}^k\binom{k}{j}(\eps t)^j
 \!\!\!\sum_{\substack{r_0+\cdots+r_k=n-k-2j\\r_i\ge0}}
 Z_{r_0}^{(\eps)}\cdots Z_{r_k}^{(\eps)}.
\end{equation}
For fixed $k$,
\begin{equation}\label{eq:penta-fixedpoint-layer-gf}
 \sum_{n\ge0}[c^k]F_n^{(\eps)}(c)z^n
 =\frac{z^k(1-tz^2)^{k+1}(1+\eps tz^2)^k}
 {Q_{\eps}(z)^{k+1}}.
\end{equation}
In $\mathbb Q(p,t,v)[z]$,
\begin{equation}\label{eq:penta-fixedpoint-gcd}
 \gcd\bigl(Q_{\eps}(z),
 z(1-tz^2)(1+\eps tz^2)\bigr)=1.
\end{equation}
Hence for either determinant or permanent the generic minimal recurrence
degree of the fixed-$k$ sequence is
\begin{equation}\label{eq:penta-fixedpoint-generic-order}
 6(k+1).
\end{equation}
\end{corollary}

\begin{proof}
Expand \eqref{eq:penta-renewal-gf} geometrically.  Coefficient extraction
gives \eqref{eq:penta-fixedpoint-coeff} and
\eqref{eq:penta-fixedpoint-convolution}, while
\eqref{eq:penta-zero-gf-unified} gives
\eqref{eq:penta-fixedpoint-layer-gf}.
Since $Q_{\eps}(0)=1$ and $Q_{\eps}$ is affine in $v$ with coefficient
$-z^3$, no nonconstant factor of
$(1-tz^2)(1+\eps tz^2)$ can divide $Q_{\eps}$ over
$\mathbb Q(p,t,v)$.  This proves \eqref{eq:penta-fixedpoint-gcd}, so the
reduced denominator is $Q_{\eps}^{k+1}$ of degree $6(k+1)$.
\end{proof}

\begin{remark}[Near-leading layers]
For each fixed $r\ge0$, Corollary~\ref{cor:penta-fixedpoint-layers} also
implies a simple qualitative description of the coefficient at fixed distance
from the leading term.  Indeed, with
\[
 H_{\eps}(z):=(1+\eps tz^2)Z^{(\eps)}(z)=1+O(z^2),
\]
setting $k=n-r$ in \eqref{eq:penta-fixedpoint-coeff} and expanding
$H_{\eps}(z)^{n-r}$ shows that
\[
 [c^{n-r}]F_n^{(\eps)}(c)
\]
is a polynomial in $n$ of degree at most $\lfloor r/2\rfloor$.
We do not pursue explicit formulas for these near-leading layers here.
\end{remark}

\begin{corollary}[Two-sign pentadiagonal square formula]
\label{cor:penta-two-sign-square}
For $\sigma\in\{\pm1\}$ put
\[
 A_n^{(\sigma)}=T_n(a,b,0,\sigma b,-a),
 \qquad
 P_n^{(\sigma)}=\perm A_n^{(\sigma)}.
\]
Then
\begin{equation}\label{eq:penta-two-sign-odd}
 P_{2r+1}^{(\sigma)}=0
 \qquad(r\ge0),
\end{equation}
and
\begin{align}
 P_{2r}^{(\sigma)}
 &=\sigma^r U_{r+1}(b,\sigma a^2)^2 \label{eq:penta-two-sign-lucas}\\
 &=\sigma^r\left[
   \sum_{j=0}^{\lfloor r/2\rfloor}
   (-\sigma)^j\binom{r-j}{j}a^{2j}b^{r-2j}
  \right]^2. \label{eq:penta-two-sign-sum}
\end{align}
Equivalently, if
\[
 K_r^{(\sigma)}=
 \begin{pmatrix}
 b&a&&\\
 \sigma a&b&\ddots&\\
 &\ddots&\ddots&a\\
 &&\sigma a&b
 \end{pmatrix}_{r\times r},
\]
then
\begin{equation}\label{eq:penta-two-sign-continuant}
 P_{2r}^{(\sigma)}=\sigma^r\det(K_r^{(\sigma)})^2.
\end{equation}
\end{corollary}

\begin{proof}
Let $q=-\sigma$ and
$R_n=\operatorname{diag}(q,q^2,\ldots,q^n)$.  A direct check on the four
nonzero offsets gives
\[
 (A_n^{(\sigma)})^T=-R_n^{-1}A_n^{(\sigma)}R_n.
\]
The permanent is invariant under diagonal similarity, so
$P_n^{(\sigma)}=(-1)^nP_n^{(\sigma)}$, proving
\eqref{eq:penta-two-sign-odd} in characteristic different from $2$; the
polynomial identity then holds integrally.

For the even subsequence $E_r=P_{2r}^{(\sigma)}$, specialization of
\eqref{eq:pentaperm} gives
\begin{equation}\label{eq:penta-two-sign-even-rec}
 E_{r+3}=(\sigma b^2-a^2)E_{r+2}
 +(a^4-\sigma a^2b^2)E_{r+1}+a^6E_r,
\end{equation}
with
\[
 E_0=1,\qquad E_1=\sigma b^2,\qquad
 E_2=(b^2-\sigma a^2)^2.
\]
Put $q_r=U_{r+1}(b,\sigma a^2)$, so
$q_r=bq_{r-1}-\sigma a^2q_{r-2}$.  On the open set where
$z^2-bz+\sigma a^2$ has distinct roots $\alpha,\beta$,
$q_r=(\alpha^{r+1}-\beta^{r+1})/(\alpha-\beta)$.  Hence
$\sigma^rq_r^2$ is a linear combination of the three modes
\[
 \sigma\alpha^2,\qquad a^2,\qquad\sigma\beta^2,
\]
whose characteristic polynomial is
\begin{equation}\label{eq:penta-two-sign-even-char}
 (t-a^2)\bigl[t^2-(\sigma b^2-2a^2)t+a^4\bigr].
\end{equation}
It therefore satisfies \eqref{eq:penta-two-sign-even-rec} and has the same
three initial values.  The identity extends polynomially to the exceptional
parameter locus, proving \eqref{eq:penta-two-sign-lucas}.  The standard
Lucas expansion gives \eqref{eq:penta-two-sign-sum}.  Finally,
$\det K_r^{(\sigma)}$ obeys the same second-order recurrence and initial
conditions as $U_{r+1}(b,\sigma a^2)$, proving
\eqref{eq:penta-two-sign-continuant}.
\end{proof}

\begin{remark}[Factorized recurrence and generic degree]
\label{rem:penta-two-sign-recurrence}
Equation~\eqref{eq:penta-two-sign-even-char} shows that the full sequence
$P_n^{(\sigma)}$ is annihilated by
\begin{equation}\label{eq:penta-two-sign-full-char}
 \chi_\sigma(x)
 =(x^2-a^2)\bigl[x^4-(\sigma b^2-2a^2)x^2+a^4\bigr].
\end{equation}
Over characteristic zero, if
\[
 a\,b\,(b^2-4\sigma a^2)\ne0,
\]
the three even-step modes in the proof are distinct and all occur with
nonzero coefficients.  Thus the even subsequence has minimal recurrence
degree $3$, and the zero-interlaced full sequence has minimal degree $6$;
on this Zariski-open locus \eqref{eq:penta-two-sign-full-char} is the scalar
minimal polynomial.  The six-state transfer characteristic polynomial
specializes to the same factorization; on the displayed open locus the Krylov
scalarization therefore returns it as the minimal polynomial.
\end{remark}

\begin{remark}[Pentadiagonal specialization and the signed-square mechanism]
\label{rem:penta-skew-conversion}
For $m=2$, Corollary~\ref{cor:penta-two-sign-square} contains both the skew
family $T_n(a,b,0,-b,-a)$ and its companion $T_n(a,b,0,b,-a)$.  The
$\sigma=-1$ case is the genuinely skew-symmetric member and reads
\begin{equation}\label{eq:skew-penta-perm-cheb}
 \perm T_{2r}(a,b,0,-b,-a)
 =(-1)^r U_{r+1}(b,-a^2)^2
 =(-1)^r\left(
   \sum_{j=0}^{\lfloor r/2\rfloor}
   \binom{r-j}{j}a^{2j}b^{r-2j}
  \right)^2.
\end{equation}
Elouafi \cite{Elouafi2011} proved, in Chebyshev notation,
\[
 \det T_{2r}(a,b,0,-b,-a)
 =\left[a^r\mathsf U_r\!\left(\frac{b}{2a}\right)\right]^2
 \qquad(a\ne0),
\]
with polynomial continuation at $a=0$.  Thus the Lucas formula above is
consistent with the classical closed determinant formula while keeping the
paper's notation $U_n(P,Q)$ distinct from the Chebyshev polynomial
$\mathsf U_n$.

The square structure has a conceptual explanation.  For $\sigma=+1$, the
real P\'olya converter sends the permanent to the determinant of the skew
matrix $T_n(a,b,0,-b,-a)$.  For $\sigma=-1$, the phase converter sends it to
the determinant of the skew matrix $T_n(-a,ib,0,-ib,a)$.  Hence in both
cases
\[
 \text{permanent}\;\longrightarrow\;
 \text{skew determinant}\;\longrightarrow\;
 \text{Pfaffian square},
\]
which explains the signed squares in
Corollary~\ref{cor:penta-two-sign-square}.
\end{remark}

\begin{remark}[The signed-square permanent phenomenon stops at semibandwidth three]
\label{rem:permanent-square-fails}
The preceding permanent-square structure is special to the exceptional
pentadiagonal conversion and does not extend to general skew Toeplitz bands.
For comparison, on the skew semibandwidth-three family
$T_n(1,2,b,0,-b,-2,-1)$, Proposition~\ref{prop:skew-toeplitz-square} gives
the natural signed square roots
\[
 b,\qquad b^2+b-4,\qquad b^3+2b^2-8b+3,\qquad
 b^4+3b^3-11b^2-2b+12
\]
for determinant orders $2,4,6,8$, respectively.  The permanent already
breaks the square pattern at order $4$:
\begin{equation}\label{eq:band3-permanent-not-square}
 \perm T_4(1,2,b,0,-b,-2,-1)
 =b^4-2b^3+9b^2+8b+16.
\end{equation}
This polynomial is not a square in $\mathbb Q[b]$: if it were
$(b^2+ub+v)^2$, comparison of the $b^3$ and $b^2$ coefficients would give
$u=-1$ and $v=4$, whereas the $b$ coefficient would then be $-8$ rather
than $8$.  It is not the negative of a square either, by its leading
coefficient.  Thus the odd-order vanishing of skew-symmetric permanents
persists for every semibandwidth, but the even-order signed-square property does
not.
\end{remark}

\begin{remark}[Recurrence check and scope]\label{rem:penta-conversion-recurrence}
The Toeplitz identity is also visible directly in the sixth-degree formulas:
setting $a_0=0$ in Sweet's determinant recurrence \cite{Sweet} and replacing $a_1$ by $-a_1$ turns its coefficients into those of \eqref{eq:pentaperm}, with the corresponding initial values transformed in the same way; see also the companion derivation \cite{AlekseyevKhomovskyRecurrences2026}.  The graph proof is stronger:
it shows that Toeplitz constancy is unnecessary and identifies the identity
as a P\'olya--Kasteleyn conversion on an entire coordinate subspace.  The
phase form \eqref{eq:penta-phase-conversion} is a gauge-equivalent complex
weighting; Kasteleyn weightings with the same curvature differ by vertex
gauges, represented on the bipartite matrix by diagonal multiplications
\cite{Kuperberg1998}.  This also places the fourth-root-of-unity substitution
within the broader literature on complex and Schur-multiplier conversions
\cite{HwangKimSong1996,Fahssi2026}.
\end{remark}

The preceding conversion is exceptional already at the level of the full
coordinate support.  For $\ell,u\ge1$ put
\[
 S_{\ell,u}=\{-\ell,\ldots,-1,1,\ldots,u\},
\]
and let $\mathcal B_{\ell,u}^{(n)}$ be the coordinate subspace of
$n\times n$ matrices $B=(b_{ij})$ satisfying
\[
 b_{ij}=0\qquad\text{unless}\qquad j-i\in S_{\ell,u}.
\]

\begin{theorem}[Entrywise converters for consecutive zero-diagonal band supports]
\label{thm:consecutive-support-converter}
Let $\ell,u\ge1$.  There exists, for every $n\ge0$, a multiplier matrix
$C_n=(c_{ij}^{(n)})$ with nonzero complex entries on the support
$j-i\in S_{\ell,u}$ such that
\begin{equation}\label{eq:consecutive-support-converter}
 \perm B=\det(C_n\circ B)
 \qquad\text{for every }B\in\mathcal B_{\ell,u}^{(n)}
\end{equation}
if and only if
\[
 \min(\ell,u)=1
 \qquad\text{or}\qquad
 (\ell,u)=(2,2).
\]
In every positive case the multipliers may be chosen in $\{\pm1\}$ by a
rule independent of $n$.
\end{theorem}

\begin{proof}
Assume first that $u=1$.  Put $c_{i,i+1}=-1$ on the first superdiagonal and
$c_{ij}=1$ on every other allowed position.  Consider one cycle of a
permutation contributing a nonzero monomial.  If $r$ is the smallest vertex
of the cycle, then $r$ cannot move to the left, so it must map to $r+1$.
Continuing around the cycle, every edge is forced to be a first-superdiagonal
edge until the unique closing edge jumps from some $s$ back to $r$.  Thus the
cycle is
\[
 (r\ r+1\ \cdots\ s),
\]
has permutation sign $(-1)^{s-r}$, and uses exactly $s-r$
first-superdiagonal edges.  The multiplier product on the cycle is therefore
also $(-1)^{s-r}$.  Multiplying over the disjoint cycles shows that every
permanent monomial receives determinant coefficient $+1$, proving
\eqref{eq:consecutive-support-converter}.  The case $\ell=1$ follows by
transposition.  The case $(\ell,u)=(2,2)$ is
Proposition~\ref{prop:penta-polya}.

Conversely, suppose $\ell,u\ge2$ and $\max(\ell,u)\ge3$.  By
transposition we may assume $\ell\ge3$ and $u\ge2$.  If a converter
existed, then by setting all entries outside the offsets
$\{-3,-2,-1,1,2\}$ equal to zero we would obtain a converter for the
$(3,2)$ support.  We show that this is impossible already for $n=6$.

For an allowed permutation $\pi\in\mathfrak S_6$, write
\[
 M_\pi(x)=\prod_{i=1}^6 x_{i,\pi(i)}.
\]
Because the supported matrix entries are algebraically independent,
coefficient comparison in \eqref{eq:consecutive-support-converter} forces
\begin{equation}\label{eq:entrywise-permutation-sign-condition}
 \prod_{i=1}^6 c_{i,\pi(i)}=\operatorname{sgn}(\pi)
\end{equation}
for every allowed permutation $\pi$.  Consider the six permutations
\begin{align*}
 \pi_1&=(2,1,4,3,6,5),&
 \pi_2&=(2,1,4,5,6,3),&
 \pi_3&=(2,1,5,3,6,4),\\
 \pi_4&=(2,1,5,6,4,3),&
 \pi_5&=(2,3,1,5,6,4),&
 \pi_6&=(2,3,1,6,4,5).
\end{align*}
All their displacements lie in $\{-3,-2,-1,1,2\}$, and their signs are
\[
 -1,+1,+1,+1,+1,+1,
\]
respectively.  Their matching monomials satisfy the exact relation
\begin{equation}\label{eq:six-matching-toric-relation}
 M_{\pi_2}(x)M_{\pi_3}(x)M_{\pi_6}(x)
 =M_{\pi_1}(x)M_{\pi_4}(x)M_{\pi_5}(x).
\end{equation}
Indeed, both sides equal
\[
\begin{aligned}
 &x_{12}^3x_{21}^2x_{23}x_{31}x_{34}x_{35}\\
 &\qquad{}\times x_{43}x_{45}x_{46}x_{54}x_{56}^2\\
 &\qquad{}\times x_{63}x_{64}x_{65}.
\end{aligned}
\]
Substituting the multiplier variables $c_{ij}$ into
\eqref{eq:six-matching-toric-relation} and using
\eqref{eq:entrywise-permutation-sign-condition}, the left-hand side must
have value
\[
 \operatorname{sgn}(\pi_2)\operatorname{sgn}(\pi_3)
 \operatorname{sgn}(\pi_6)=+1,
\]
whereas the right-hand side must have value
\[
 \operatorname{sgn}(\pi_1)\operatorname{sgn}(\pi_4)
 \operatorname{sgn}(\pi_5)=-1,
\]
a contradiction.  Thus no entrywise converter exists for the $(3,2)$
support, and hence none exists in any larger consecutive two-sided support.
\end{proof}

\begin{remark}[Support-level versus Toeplitz conversion]
\label{rem:support-versus-toeplitz-conversion}
The obstruction above is stronger than a failure of P\'olya signing: it uses
only the multiplicative relation among six perfect matchings and therefore
rules out arbitrary nonzero complex entrywise multipliers.  The theorem
concerns all matrices on the given coordinate support.  The next result asks
a different question: after restricting to the Toeplitz locus, can one
convert by multipliers that are constant along each diagonal?  The answer is
again the same exceptional list, but the negative direction requires a
separate Toeplitz-monomial argument.
\end{remark}

The preceding support-level classification allows arbitrary entrywise
multipliers.  We now impose Toeplitz constancy on the matrices and require the
allowed modification to be constant along each Toeplitz diagonal.  The next
result gives the exact classification in that setting.

\begin{theorem}[Diagonal-wise converters for consecutive zero-diagonal bands]
\label{thm:consecutive-converter}
Let $\ell,u\ge1$.  There exist constants $\omega_s\in\mathbb C$, one for each
$s\in S_{\ell,u}$, such that
\begin{equation}\label{eq:consecutive-converter}
 \perm T_n((a_s)_{s\in S_{\ell,u}})
 =\det T_n((\omega_s a_s)_{s\in S_{\ell,u}})
\end{equation}
for every $n\ge0$ and every choice of the diagonal parameters if and only if
\[
 \min(\ell,u)=1
 \qquad\text{or}\qquad
 (\ell,u)=(2,2).
\]
Here the main diagonal is zero in both matrices.
\end{theorem}

\begin{proof}
If $u=1$, the companion paper \cite{AlekseyevKhomovskyRecurrences2026} gives, with the usual zero initial convention,
\[
 D_n=\sum_{r=0}^{\ell}(-1)^r a_{-r}a_1^rD_{n-r-1},\qquad
 P_n=\sum_{r=0}^{\ell}a_{-r}a_1^rP_{n-r-1}.
\]
Replacing $a_1$ by $-a_1$ in the determinant recurrence gives the permanent recurrence. Hence one may take $\omega_1=-1$ and $\omega_{-r}=1$ for $1\le r\le \ell$.  The case
$\ell=1$ follows by transposition.  Notice that these Hessenberg conversions do
not in fact require the main diagonal to vanish.  The case $(\ell,u)=(2,2)$ is
Corollary~\ref{cor:penta-real-conversion}.

Conversely, suppose $\ell,u\ge2$ and $\max(\ell,u)\ge3$.  By transposition we may
assume $\ell\ge3$ and $u\ge2$.  Since \eqref{eq:consecutive-converter} is
required for every choice of the parameters, setting all diagonals outside
$\{-3,-2,-1,1,2\}$ equal to zero would give such a converter for the
$(3,2)$ band.  We show that this is impossible already at $n=8$.

For a permutation $\pi\in\mathfrak S_8$, its Toeplitz monomial is
$\prod_{i=1}^8 a_{\pi(i)-i}$.  The monomial
\begin{equation}\label{eq:converter-obstruction-monomial}
 M=a_{-3}a_{-2}^{\,2}a_{-1}a_2^{\,4}
\end{equation}
is produced by exactly the following four permutations, written in one-line
notation:
\begin{align*}
 &(3,1,5,2,7,8,4,6),
 &&(3,4,1,2,7,8,6,5),\\
 &(3,4,1,6,2,8,5,7),
 &&(3,4,2,1,7,8,5,6).
\end{align*}
Their signs are respectively $+,-,+,-$.  Thus the coefficient of $M$ in the
permanent is $4$.  After any diagonal-wise rescaling $a_s\mapsto\omega_sa_s$,
all four determinant contributions acquire the same factor
\[
 \omega_{-3}\omega_{-2}^{\,2}\omega_{-1}\omega_2^{\,4},
\]
so their coefficient is still
\[
 \omega_{-3}\omega_{-2}^{\,2}\omega_{-1}\omega_2^{\,4}
 (1-1+1-1)=0.
\]
This contradicts \eqref{eq:consecutive-converter} and proves the negative
direction.
\end{proof}

The obstruction remains after imposing skew-symmetry and even if the phase
change is allowed to destroy that symmetry.

\begin{corollary}[No diagonal-wise skew conversion beyond semibandwidth two]
\label{cor:skew-no-converter}
Let $m\ge3$.  There are no constants $\omega_{\pm1},\ldots,\omega_{\pm m}$
such that
\[
 \perm T_n(a_{-m},\ldots,a_{-1},0,a_1,\ldots,a_m)
 =\det T_n(\omega_{-m}a_{-m},\ldots,\omega_m a_m)
\]
for every $n$ and every skew-symmetric specialization
$a_{-s}=-a_s$.  This remains true with the multipliers on the positive and
negative diagonals chosen independently.
\end{corollary}

\begin{proof}
It is enough to treat $m=3$, since all more distant diagonals may be set to
zero.  At $n=8$, after imposing $a_{-s}=-a_s$, consider the coefficient of
\[
 a_1a_2^{\,6}a_3.
\]
The eight contributing permutations are exactly the four permutations in
the proof of Theorem~\ref{thm:consecutive-converter} and their inverses.
Each has an even number of negative displacements, so all eight contributions
to the permanent have sign $+1$ and the coefficient is $8$.  In the
determinant, the first four have one common multiplier monomial and cancel
by their permutation signs $+,-,+,-$; their inverses have the transposed
common multiplier monomial and cancel separately.  Hence the determinant
coefficient is identically zero for arbitrary independent multipliers
$\omega_{\pm1},\omega_{\pm2},\omega_{\pm3}$.
\end{proof}

\begin{remark}[The recurrence-degree gap begins at semibandwidth three]
\label{rem:converter-order-gap}
For the balanced skew-symmetric open family of semibandwidth $m$, the determinant sequence has
generic minimal degree $2\cdot3^{m-1}$ by Theorem~\ref{thm:skew}, whereas
the full permanent transfer gives the upper bound $\binom{2m}{m}$; after
passing to the even-size subsequences these become, respectively,
$3^{m-1}$ and $\frac12\binom{2m}{m}$.  The two bounds agree for $m=1,2$
and separate thereafter.  Indeed,
\[
 \frac{\binom{2m+2}{m+1}}{\binom{2m}{m}}
 =4-\frac{2}{m+1}>3\qquad(m\ge2),
\]
while $2\cdot3^{m-1}$ grows by the factor $3$.  Thus
\[
 \binom{2m}{m}>2\cdot3^{m-1}\qquad(m\ge3).
\]
This comparison alone would not prove nonconvertibility, because the
permanent bound is not asserted to be generically minimal for all $m$.
Corollary~\ref{cor:skew-no-converter} supplies the direct algebraic
obstruction.
\end{remark}

\section{Permanent symmetry quotients}\label{sec:permanent-symmetry}

The determinant reductions above arise from collisions among Widom
root-product modes.  For permanents no analogue of that root-product formula
is needed: the exact row-column state space itself carries the relevant
symmetry.

\begin{lemma}[Transpose of balanced row-column signatures]\label{lem:signature-transpose}
Assume $m_1=m_2=m$, and write $\operatorname{norm}$ for the normalization
of the boundary-minor signatures recalled in Proposition~\ref{prop:state-recall}.  For
$A,B\in\binom{[m]}j$,
\begin{equation}\label{eq:signature-transpose}
 \operatorname{norm}\!\left(\Sigma(A,B)^{\mathsf T}\right)=\Sigma(B,A).
\end{equation}
\end{lemma}

\begin{proof}
Use the inverse description \eqref{eq:AB-from-signature}.  Before the final
re-encoding by \eqref{eq:sigmaAB}, transposition simply interchanges the row
and column deficit sets.  The normalization rule is symmetric in those two
sets: whenever both contain the leading deficit $1$, it removes that common
deficit and translates the remaining deficits simultaneously.  Hence applying
\eqref{eq:AB-from-signature} after transposition and normalization returns the
ordered pair $(B,A)$.  Since \eqref{eq:sigmaAB} is inverse to
\eqref{eq:AB-from-signature}, \eqref{eq:signature-transpose} follows.  The
$\widehat A$ convention in \eqref{eq:sigmaAB} is precisely the endpoint
correction that records this normalization when a deficit passes through
$m$.
\end{proof}

\begin{theorem}[Symmetric open permanent bound]\label{thm:per-sym-open}
Let $m_1=m_2=m\ge1$ and assume $a_{-s}=a_s$ for $1\le s\le m$.  If
$P_n=\perm A_n$, then $(P_n)$ has a homogeneous
constant-coefficient annihilator of degree at most
\begin{equation}\label{eq:per-sym-open}
 d_{\mathrm{open}}^{\mathrm{per},+}
 =\frac12\left(\binom{2m}{m}+2^m\right).
\end{equation}
\end{theorem}

\begin{proof}
At exchange level $j$, Proposition~\ref{prop:state-recall} identifies the balanced
row-column states with pairs of $j$-subsets of $[m]$.  By
Lemma~\ref{lem:signature-transpose}, transposition acts on the normalized
labels exactly as
\[
 (A,B)\longleftrightarrow(B,A).
\]
Hence the level-$j$ involution has $\binom mj^2$ states in total and
$\binom mj$ fixed states, namely those with $A=B$.

For a symmetric Toeplitz matrix, transposed boundary minors have the same
permanent because $A_n^{\mathsf T}=A_n$ and
$\perm M=\perm M^{\mathsf T}$.  The row-column state
vector therefore lies, for every $n$, in the fixed space of this explicit
transposition involution.  Its dimension is the number of involution orbits, namely
\[
 \frac12\sum_{j=0}^m
 \left(\binom mj^2+\binom mj\right)
 =\frac12\left(\binom{2m}{m}+2^m\right),
\]
by Vandermonde's identity and $\sum_j\binom mj=2^m$.  The Krylov orbit of the
principal state is contained in this fixed space, so a linear dependence
among at most $d_{\mathrm{open}}^{\mathrm{per},+}+1$ consecutive state
vectors propagates under the constant transfer and yields the stated scalar
annihilator for the principal permanent coordinate.
\end{proof}

For $m=2$ the upper bound is already symbolically sharp.  Indeed, with
$a_{-1}=a_1$ and $a_{-2}=a_2$, direct evaluation of the first nine
pentadiagonal permanents gives
\begin{equation}\label{eq:per-hankel-m2}
 H^{\mathrm{per}}_5
 \defeq\det(P_{i+j})_{0\le i,j<5}
 =-2a_1^2a_2^{18}.
\end{equation}
Thus the symmetric pentadiagonal permanent has generic minimal degree $5$ on
the genuine-band locus.

\begin{corollary}[Exact symmetric pentadiagonal exceptional locus]\label{cor:penta-exceptional}
Assume $a_{-2}=a_2$, $a_{-1}=a_1$, and $a_2\ne0$.  Let
$X_n^{(\eps)}$ denote the symmetric pentadiagonal determinant when
$\eps=-1$ and the corresponding permanent when $\eps=+1$.  Then
\[
 \deg_{\min} X^{(\eps)}=
 \begin{cases}
  5,&a_1\ne0,\\
  4,&a_1=0.
 \end{cases}
\]
On the exceptional hyperplane $a_1=0$, put
\[
 F_0=1,\qquad F_1=a_0,\qquad
 F_{k+2}=a_0F_{k+1}+\eps a_2^2F_k.
\]
Then
\begin{equation}\label{eq:penta-parity-split}
 X_{2k}^{(\eps)}=F_k^2,\qquad
 X_{2k+1}^{(\eps)}=F_kF_{k+1},
\end{equation}
and
\begin{equation}\label{eq:penta-drop-rec}
 X_{n+4}^{(\eps)}
 =a_0X_{n+3}^{(\eps)}
  +\eps a_0a_2^2X_{n+1}^{(\eps)}
  +a_2^4X_n^{(\eps)}.
\end{equation}
Equivalently,
\begin{equation}\label{eq:penta-drop-gf}
 \sum_{n\ge0}X_n^{(\eps)}z^n
 =\frac{1}{1-a_0z-\eps a_0a_2^2z^3-a_2^4z^4}.
\end{equation}
Thus, on the genuine symmetric pentadiagonal locus, $a_1=0$ is exactly a
one-step recurrence-degree drop for both determinants and permanents.
\end{corollary}

\begin{proof}
When $a_1=0$, simultaneously reorder the rows and columns by parity.  The
matrix becomes block diagonal with two symmetric tridiagonal Toeplitz blocks;
for size $2k$ the block sizes are $k,k$, and for size $2k+1$ they are
$k+1,k$.  The same row and column permutation leaves both determinant and
permanent unchanged.  The determinant of a tridiagonal block satisfies the
second-order recurrence above with $\eps=-1$, while its permanent satisfies
it with $\eps=+1$.  This proves \eqref{eq:penta-parity-split}.

A direct substitution of the second-order recurrence for $F_k$ into the two
parities in \eqref{eq:penta-parity-split} gives
\eqref{eq:penta-drop-rec}.  The initial values
\[
 X_0^{(\eps)}=1,\quad X_1^{(\eps)}=a_0,\quad
 X_2^{(\eps)}=a_0^2,\quad
 X_3^{(\eps)}=a_0^3+\eps a_0a_2^2
\]
then give \eqref{eq:penta-drop-gf}; all numerator correction terms cancel.
Because $a_2\ne0$, the denominator in \eqref{eq:penta-drop-gf} has degree
four, and its numerator is $1$, so no cancellation is possible.  Hence the
minimal degree is exactly $4$ on $a_1=0$.

If $a_1\ne0$, the determinant Hankel factorization in
Remark~\ref{rem:exceptional-locus} and the permanent factorization
\eqref{eq:per-hankel-m2} are both nonzero.  The corresponding degree-$5$
upper bounds therefore become exact.  This proves the stated dichotomy.
\end{proof}

\begin{proposition}[Skew-symmetric open permanents]\label{prop:per-skew-open}
Assume the coefficient field has characteristic different from $2$, let
$m\ge1$, and impose $a_0=0$ and $a_{-s}=-a_s$.  Then
\[
 P_{2k+1}=0.
\]
If $E_k^{\mathrm{per}}\defeq P_{2k}$, then the even-size permanent subsequence
has a homogeneous constant-coefficient annihilator of degree at most
\begin{equation}\label{eq:per-skew-open-even}
 \frac12\binom{2m}{m}.
\end{equation}
\end{proposition}

\begin{proof}
Since $A_n^{\mathsf T}=-A_n$,
\[
 P_n=\perm A_n^{\mathsf T}
 =\perm(-A_n)=(-1)^nP_n,
\]
which proves the odd-size vanishing.  By Proposition~\ref{prop:state-recall}, the full
permanent sequence has a rational generating function $F(z)=N(z)/D(z)$ with
$D(0)=1$ and $\deg D\le\binom{2m}{m}$.  Take $N,D$ coprime.  Since only even
terms survive, $F(-z)=F(z)$.  Coprimeness and the normalization $D(0)=1$
then force $D(-z)=D(z)$, so $D(z)=R(z^2)$ with
$\deg R\le\frac12\binom{2m}{m}$.  Writing
$F(z)=G(z^2)$ shows that $R$ is an annihilating denominator for the generating
function of $(E_k^{\mathrm{per}})$.
\end{proof}

The two permanent bounds above concern state-space quotients and parity,
rather than the determinant root-product collisions of the preceding
subsections.  Section~\ref{sec:cyclic-symmetry} derives the cyclic counterparts
by viewing a banded circulant as a one-sided Toeplitz band with a fixed corner
defect.  Resolving that defect produces a finite family of clean but shifted
central bands, so the cyclic bounds reduce directly to the ordinary
row--column state counts of the companion recurrence paper.

\section{Cyclic symmetry quotients}\label{sec:cyclic-symmetry}

For cyclic closure we use the notation of the companion paper.  Let
\[
 C_n=a(\Pi_n),\qquad n>2m,
\]
where $\Pi_n$ is the cyclic shift.  If $q(z)=a_m\prod_{j=1}^{2m}(z-z_j)$, then \cite{AlekseyevKhomovskyRecurrences2026}
\begin{equation}\label{eq:cyclic-det-product-recalled}
 \det C_n=(-1)^{m(n+1)}a_m^n\prod_{j=1}^{2m}(1-z_j^n).
\end{equation}
The unrestricted cyclic determinant uses all $2^{2m}$ subset modes.  Balanced symmetry collapses this full Boolean spectrum in a particularly simple way.  Before specializing that symmetry, we record an exact bridge in the opposite direction: an open banded Toeplitz determinant can be recovered from a slightly larger circulant and a correction determinant whose size is only the semibandwidth.  On the permanent side, recurrence questions for circulants go back at least to Minc's work on $(0,1)$-circulants \cite{Minc1964}; later sparse-circulant determinant representations provide a complementary conversion viewpoint \cite{CodenottiResta2002}.

\subsection{A Toeplitz--circulant bridge}\label{subsec:toeplitz-circulant-bridge}

Fixed-size reductions of banded Toeplitz determinants are classical.  In
particular, Baxter and Schmidt reduce such determinants to determinants of
fixed order formed from coefficients of the reciprocal symbol
\cite{BaxterSchmidt1961}; see also the determinant treatment in
\cite[Chapter~2]{BottcherGrudsky2005}.  Circulant embeddings are likewise a
standard companion to Toeplitz matrices, while Jacobi's complementary-minor
identity is a well-established tool for exact Toeplitz determinant identities
\cite{BottcherWidom2006}.  The observation below combines these viewpoints in
a form tailored to the present open--cyclic comparison: the balanced open
section is a principal block of a minimally larger circulant, and its
determinant is obtained from the circulant determinant by a fixed-size inverse
block correction.

Let
\[
 T_n(a)=(a_{j-i})_{i,j=1}^n,
 \qquad
 a(z)=\sum_{s=-m}^m a_s z^s,
\]
and put $D_n=\det T_n(a)$.  For $n>m$ set
\[
 N=n+m,
 \qquad
 \widehat C_N=a(\Pi_N).
\]
Since $N>2m$, no wrap-around entry of $\widehat C_N$ lies inside its
leading $n\times n$ principal block, and that block is exactly $T_n(a)$.
Thus the open Toeplitz section is obtained from a minimal circulant completion
by deleting only $m$ rows and the same $m$ columns.

\begin{proposition}[Fixed-size circulant correction]\label{prop:toeplitz-circulant-correction}
Let $n>m$, $N=n+m$, and assume first that $\widehat C_N$ is nonsingular.
If
\[
 S=\{n+1,\ldots,N\},
\]
then
\begin{equation}\label{eq:toeplitz-circulant-jacobi}
 D_n
 =\det(\widehat C_N)\,
  \det\!\bigl(\widehat C_N^{-1}[S,S]\bigr).
\end{equation}
Moreover $\widehat C_N^{-1}$ is circulant, and if
\begin{equation}\label{eq:circulant-inverse-coefficients}
 g_r^{(N)}
 =\frac1N\sum_{\omega^N=1}\frac{\omega^{-r}}{a(\omega)},
 \qquad |r|\le m-1,
\end{equation}
then
\begin{equation}\label{eq:toeplitz-circulant-fourier}
 D_n
 =\left(\prod_{\omega^N=1}a(\omega)\right)
   \det\bigl(g_{j-i}^{(N)}\bigr)_{i,j=1}^m.
\end{equation}
Thus the correction determinant has order $m$, independent of $n$.
\end{proposition}

\begin{proof}
The leading $n\times n$ principal block of $\widehat C_N$ is $T_n(a)$.
Jacobi's complementary-minor identity therefore gives
\eqref{eq:toeplitz-circulant-jacobi}.  The Fourier vectors diagonalize
$\widehat C_N$ with eigenvalues $a(\omega)$, $\omega^N=1$.  Hence its
inverse is circulant with offset-$r$ coefficient
\eqref{eq:circulant-inverse-coefficients}, while
\[
 \det\widehat C_N=\prod_{\omega^N=1}a(\omega).
\]
Restricting the inverse to the consecutive index set $S$ gives the
$m\times m$ Toeplitz block
$(g_{j-i}^{(N)})$, proving \eqref{eq:toeplitz-circulant-fourier}.
The identities extend across singular specializations after clearing the
inverse denominators.
\end{proof}

The correction coefficients also make the relation with the classical
reciprocal-symbol approach explicit.  If an annulus containing the unit circle
supports a Laurent expansion
\[
 \frac1{a(z)}=\sum_{k\in\mathbb Z}h_k z^k,
\]
then discrete Fourier orthogonality in
\eqref{eq:circulant-inverse-coefficients} gives
\begin{equation}\label{eq:circulant-aliasing}
 g_r^{(N)}=\sum_{\ell\in\mathbb Z}h_{r+\ell N}.
\end{equation}
Thus the inverse-circulant correction uses the $N$-periodization, or aliasing,
of the reciprocal-symbol coefficients that appear in fixed-size Toeplitz
reductions such as Baxter--Schmidt.

The fixed-size correction also gives a circulant-completion route from the
elementary circulant product to the Widom expansion recalled in
\eqref{eq:widomform}.  This route parallels the classical Baxter--Schmidt
derivation, but uses periodic Fourier coefficients of the inverse symbol in
place of the unperiodized reciprocal coefficients.  We first record the
elementary middle-subset interpolation identity that will convert the
Cauchy--Binet expansion of the correction into the Widom sum.

\begin{lemma}[Middle-subset interpolation identity]\label{lem:middle-subset-interpolation}
Let $z_1,\ldots,z_{2m}$ be pairwise distinct and define
\[
 \Delta_J^{\times}
 =\prod_{\substack{j\in J\\k\notin J}}(z_j-z_k),
 \qquad |J|=m.
\]
Then for arbitrary $y_1,\ldots,y_{2m}$,
\begin{equation}\label{eq:middle-subset-interpolation}
 \sum_{|J|=m}
 \frac{\displaystyle\prod_{k\notin J}(1-y_k)}{\Delta_J^{\times}}
 =
 \sum_{|J|=m}
 \frac{\displaystyle\prod_{j\in J}y_j}{\Delta_J^{\times}}.
\end{equation}
\end{lemma}

\begin{proof}
Put
\[
 F(t)=\sum_{|J|=m}
 \frac{\prod_{j\in J}(t-y_j)}{\Delta_J^{\times}}.
\]
We show that $F$ is constant.  Differentiating and grouping terms by an
$(m-1)$-subset $H$ gives
\[
 F'(t)=\sum_{|H|=m-1}\prod_{h\in H}(t-y_h)
 \sum_{\ell\notin H}\frac1{\Delta_{H\cup\{\ell\}}^{\times}}.
\]
Let $L=[2m]\setminus H$, so $|L|=m+1$, and put
$p_L(z)=\prod_{\ell\in L}(z-z_\ell)$ and
$P_H(z)=\prod_{h\in H}(z-z_h)$.  Up to a nonzero factor independent of
$\ell$, the inner sum is
\[
 \sum_{\ell\in L}\frac{P_H(z_\ell)}{p_L'(z_\ell)}.
\]
This is zero by Lagrange interpolation because
$\deg P_H=m-1<|L|-1=m$.  Hence $F'(t)=0$.

Now $\Delta_{J^c}^{\times}=(-1)^m\Delta_J^{\times}$.  Replacing $J$ by
$J^c$ on the left of \eqref{eq:middle-subset-interpolation} shows that its
left-hand side is $(-1)^mF(1)$, whereas its right-hand side is
$(-1)^mF(0)$.  Since $F(1)=F(0)$, the identity follows.
\end{proof}

\begin{corollary}[Circulant completion recovers Widom]\label{cor:circulant-recovers-widom}
Assume that
\[
 q(z)=z^ma(z)=a_m\prod_{\nu=1}^{2m}(z-z_\nu)
\]
has simple roots and that $z_\nu^N\ne1$ for $N=n+m$.  Then, for
$|r|\le m-1$,
\begin{equation}\label{eq:circulant-correction-root-form}
 g_r^{(N)}
 =\sum_{\nu=1}^{2m}
 \frac{z_\nu^{\,m-r-1}}
 {q'(z_\nu)(1-z_\nu^N)}.
\end{equation}
Consequently
\begin{equation}\label{eq:circulant-correction-cauchy-binet}
 \det\bigl(g_{j-i}^{(N)}\bigr)_{i,j=1}^m
 =a_m^{-m}
 \sum_{|J|=m}
 \frac{1}{\Delta_J^{\times}\prod_{j\in J}(1-z_j^N)},
\end{equation}
and
\begin{equation}\label{eq:open-from-cyclic-factors}
 D_n
 =(-1)^{m(N+1)}a_m^n
 \sum_{|J|=m}
 \frac{\prod_{k\notin J}(1-z_k^N)}{\Delta_J^{\times}}.
\end{equation}
These formulas reduce to the Widom expansion
\begin{equation}\label{eq:widom-from-circulant}
 D_n
 =\sum_{|J|=m}
 \frac{\prod_{j\in J}z_j^m}{\Delta_J^{\times}}
 \left((-1)^m a_m\prod_{j\in J}z_j\right)^n.
\end{equation}
\end{corollary}

\begin{proof}[Proof of Corollary~\ref{cor:circulant-recovers-widom}]
Since $a(\omega)=\omega^{-m}q(\omega)$,
\eqref{eq:circulant-inverse-coefficients} and the partial fraction expansion
of $1/q$ reduce the required root-of-unity sum to
\[
 \frac1N\sum_{\omega^N=1}\frac{\omega^d}{\omega-z}
 =\frac{z^{d-1}}{1-z^N},
 \qquad 1\le d<N.
\]
Here $d=m-r$ lies between $1$ and $2m-1<N$, proving
\eqref{eq:circulant-correction-root-form}.

Write the correction matrix as $UDV^{\mathsf T}$, where
\[
 U_{i\nu}=z_\nu^{i-1},\qquad
 V_{j\nu}=z_\nu^{m-j},\qquad
 D_{\nu\nu}=\frac1{q'(z_\nu)(1-z_\nu^N)}.
\]
Cauchy--Binet gives a sum over $m$-subsets $J$.  The two Vandermonde
minors cancel the factors internal to $J$ in
$\prod_{j\in J}q'(z_j)$, leaving exactly
\eqref{eq:circulant-correction-cauchy-binet}.  Multiplication by the cyclic
product \eqref{eq:cyclic-det-product-recalled}, with $n$ there replaced by
$N$, gives \eqref{eq:open-from-cyclic-factors}.

Finally Lemma~\ref{lem:middle-subset-interpolation} replaces the numerator
in \eqref{eq:open-from-cyclic-factors} by $\prod_{j\in J}z_j^N$.  Since
$N=n+m$ and
\[
 m(N+1)-mn=m(m+1)
\]
is even, the resulting sign is $(-1)^{mn}$.  Splitting
$z_j^N=z_j^m z_j^n$ yields \eqref{eq:widom-from-circulant}, which is
\eqref{eq:widomform} with the coefficient formula
\eqref{eq:widom-coefficient-balanced}.
\end{proof}

\begin{remark}[Open and cyclic determinant formulas]\label{rem:open-cyclic-bridge}
Proposition~\ref{prop:toeplitz-circulant-correction} separates the passage
from cyclic to open boundary conditions into a Fourier-diagonal circulant
factor and a correction determinant of fixed order $m$.  The correction
matrix itself depends on $N=n+m$ through the inverse-circulant coefficients
\eqref{eq:circulant-inverse-coefficients}; only its order is independent of
$n$.  Corollary~\ref{cor:circulant-recovers-widom} shows that expanding this
fixed-order correction by Cauchy--Binet selects the middle-cardinality subsets
and converts the full cyclic product into the open Widom sum.  Thus the
fixed-cardinality spectrum of the open Toeplitz problem and the full Boolean
spectrum of the cyclic problem are linked by an exact fixed-order correction,
not merely by an analogy of root products.  In this sense the construction is
a circulant-completion counterpart of the classical Baxter--Schmidt route to
Widom's formula \cite{BaxterSchmidt1961,BottcherGrudsky2005}, rather than a
replacement for that fixed-size determinant theory.
\end{remark}

\begin{theorem}[Symmetric cyclic determinant reduction]\label{thm:symmetric-cyclic}
Let $m_1=m_2=m$ and $a_{-j}=a_j$.  Generically write the roots of
$q(z)=z^ma(z)$ as
\[
 x_1,x_1^{-1},\ldots,x_m,x_m^{-1}.
\]
For $n>2m$, the corresponding circulant satisfies
\begin{equation}\label{eq:sym-cyclic-product}
 \det C_n=(-1)^{m(n+1)}a_m^n
 \prod_{i=1}^m\bigl(2-x_i^n-x_i^{-n}\bigr).
\end{equation}
Equivalently, with $s(\eps)$ as in \eqref{eq:ternary-support-size},
\begin{equation}\label{eq:sym-cyclic-expansion}
 \det C_n=(-1)^m
 \sum_{\eps\in\{-1,0,1\}^m}
 2^{m-s(\eps)}(-1)^{s(\eps)}
 \left(
  (-1)^ma_m\prod_{i=1}^m x_i^{\eps_i}
 \right)^n.
\end{equation}
Hence the symmetric cyclic determinant sequence has an annihilator of degree
at most $3^m$, and its minimal recurrence degree is generically exactly
\begin{equation}\label{eq:sym-cyclic-order}
 3^m.
\end{equation}
\end{theorem}

\begin{proof}
Apply the cyclic determinant formula \eqref{eq:cyclic-det-product-recalled} from the companion paper and pair reciprocal roots:
\[
 (1-x_i^n)(1-x_i^{-n})=2-x_i^n-x_i^{-n}.
\]
Expanding the product proves \eqref{eq:sym-cyclic-expansion}.  Each reciprocal
pair now contributes one of the three exponents $-1,0,1$; the zero exponent
has coefficient $2$, while the two nonzero exponents have coefficient $-1$.
Thus all $3^m$ displayed amplitudes are nonzero.

The corresponding characteristic bases are
\[
 \lambda_{\eps}=(-1)^ma_m\prod_{i=1}^m x_i^{\eps_i},
 \qquad \eps\in\{-1,0,1\}^m.
\]
They are pairwise distinct on a nonempty Zariski-open set.  For example, take
$x_i=t^{3^{i-1}}$ with $t>1$; uniqueness of balanced ternary expansion makes
all $\lambda_{\eps}$ distinct, and
$\prod_i(z-x_i)(z-x_i^{-1})$ is palindromic, hence comes from a symmetric
Toeplitz symbol.  Lemma~\ref{lem:hankel-vandermonde} then gives a nonzero
Hankel determinant of size $3^m$ at this specialization.  Since that Hankel
determinant is polynomial in the symmetric diagonal parameters, its
nonvanishing defines a nonempty Zariski-open set on which the degree is
exactly $3^m$.
\end{proof}

For completeness, the skew-symmetric balanced cyclic case degenerates
further.  If $a_0=0$ and $a_{-j}=-a_j$, then for $n>2m$
\[
 C_n=\sum_{j=1}^m a_j(\Pi_n^j-\Pi_n^{-j}),
 \qquad C_n\mathbf 1=0,
\]
so every such circulant determinant is identically zero.

Together with Theorems~\ref{thm:symmetric} and \ref{thm:skew},
Corollary~\ref{cor:skew-even} and the general cyclic determinant result of \cite{AlekseyevKhomovskyRecurrences2026}, this gives
the compact determinant comparison
\[
\begin{array}{c|ccc}
 & \text{general balanced} & \text{symmetric} & \text{skew-symmetric}\\ \hline
 \text{open Toeplitz}
 & \binom{2m}{m} & (3^m+1)/2 & 2\cdot3^{m-1}\\[3pt]
 \text{circulant}
 & 4^m & 3^m & \text{identically }0
\end{array}
\]
for the generic minimal recurrence degrees in the nonzero determinant
families, with the lower-right entry recording the stronger zero statement.
For the skew-symmetric open family, restricting to even matrix sizes further
reduces the generic minimal degree from $2\cdot3^{m-1}$ to $3^{m-1}$ by
Corollary~\ref{cor:skew-even}.  In the open symmetric case the
fixed-cardinality condition selects one parity layer of exponent vectors;
cyclic closure admits all exponent vectors in $\{-1,0,1\}^m$.
The permanent side admits rigorous symmetry-reduced annihilator bounds as
well, but their generic minimality is a separate question.  We now derive the
cyclic bounds by collecting the wrap-around entries into a fixed corner defect
and reducing each defect sector to an ordinary shifted Toeplitz band.

For permanents it is useful to avoid introducing a separate periodic
cycle-cover transfer.  Instead we reduce cyclic closure to the ordinary
banded-matrix transfers of the companion paper by collecting the wrap-around
entries into a fixed corner defect.  This also makes transparent why the
binomial coefficients appearing below are the same state counts as for
ordinary shifted Toeplitz bands.

\begin{lemma}[Defect decomposition of a banded circulant]
\label{lem:cyclic-defect-decomposition}
Let
\[
 P_n^{\mathrm{cyc}}=\perm C_n,\qquad n>2m.
\]
There are sequences $P_{n,r}^{\mathrm{cyc}}$, $0\le r\le2m$, such that
\begin{equation}\label{eq:cyclic-defect-sum}
 P_n^{\mathrm{cyc}}=\sum_{r=0}^{2m}P_{n,r}^{\mathrm{cyc}},
\end{equation}
and, for each fixed $r$, the sequence $P_{n,r}^{\mathrm{cyc}}$ has a
homogeneous constant-coefficient annihilator of degree at most
\begin{equation}\label{eq:cyclic-defect-degree}
 \binom{2m}{r}.
\end{equation}
More precisely, after a cyclic column shift and resolution of the resulting
corner defect, every summand contributing to $P_{n,r}^{\mathrm{cyc}}$ has a
clean central band with lower and upper semibandwidths
\[
 (2m-r,r).
\]
Transposition interchanges the sectors $r$ and $2m-r$.  Hence under symmetric
weights
\begin{equation}\label{eq:cyclic-defect-sym-pair}
 P_{n,2m-r}^{\mathrm{cyc}}=P_{n,r}^{\mathrm{cyc}},
\end{equation}
whereas under $a_0=0$ and $a_{-s}=-a_s$,
\begin{equation}\label{eq:cyclic-defect-skew-pair}
 P_{n,2m-r}^{\mathrm{cyc}}=(-1)^nP_{n,r}^{\mathrm{cyc}}.
\end{equation}
\end{lemma}

\begin{proof}
Cyclically shift the columns by $m$ positions in the direction that sends the
bulk offsets $[-m,m]$ to $[-2m,0]$.  For the permanent this column
permutation does not change the value.  Away from the boundary the shifted
matrix is therefore a clean $(2m,0)$-banded Toeplitz matrix; all wrap-around
entries are confined to a fixed northeast corner involving only the last
$2m$ columns and a bounded number of initial rows.

Partition the permanent expansion according to the boundary placement of
these $2m$ defective columns.  For a fixed nonzero placement, move to the
front the $r$ defective columns whose chosen entries lie in the upper defect,
and suppress the defect entries incompatible with that placement.  This does
not change any compatible permutation term, and the placement classes are
disjoint and exhaustive.  In the central part of the resulting matrix the
movement of $r$ columns raises the clean band by $r$ positions, changing its
semibandwidths as
\[
 (2m,0)\longmapsto(2m-r,r).
\]
Only bounded-size leading and trailing corner data remain distorted.  The
number of boundary placements depends on $m$ but not on $n$.

For each fixed placement, the row--column construction of the companion
paper \cite{AlekseyevKhomovskyRecurrences2026} therefore represents the
resulting sequence as a matrix coefficient of the ordinary
$(2m-r,r)$ clean-band transfer between placement-dependent boundary vectors.
Its state dimension is
\[
 \binom{(2m-r)+r}{r}=\binom{2m}{r}.
\]
All placements with the same $r$ use this same central transfer, so their sum
is annihilated by its characteristic polynomial.  Summing over $r$ proves
\eqref{eq:cyclic-defect-sum} and \eqref{eq:cyclic-defect-degree}.

Transposition reverses the direction of every wrap-around displacement, so it
sends a placement with $r$ raised columns to the complementary placement with
$2m-r$ raised columns.  Symmetric weights preserve every monomial, giving
\eqref{eq:cyclic-defect-sym-pair}.  Under skew weights every nonzero term uses
$n$ off-diagonal entries and edge reversal changes each weight by a factor
$-1$, giving \eqref{eq:cyclic-defect-skew-pair}.
\end{proof}

\begin{remark}[Relation with winding sectors]\label{rem:defect-winding}
The shifted-band index is the same invariant that appears as winding in a
moving-cut description.  With
\[
 \delta=r-m,
\]
the clean-band state count becomes
\[
 \binom{2m}{r}=\binom{2m}{m+\delta}.
\]
Thus the defect decomposition and the periodic winding decomposition are two
coordinate descriptions of the same sector splitting.  We use the defect
version because it reduces the recurrence bounds directly to the ordinary
banded Toeplitz state counts already developed in the companion paper.
\end{remark}

\begin{theorem}[Symmetric cyclic permanent bound]\label{thm:per-sym-cyclic}
Assume $a_{-s}=a_s$ for $0\le s\le m$.  Then the stable symmetric
$(m,m)$-circulant permanent sequence has a homogeneous constant-coefficient
annihilator of degree at most
\begin{equation}\label{eq:per-sym-cyclic}
 d_{\mathrm{cyc}}^{\mathrm{per},+}
 =\frac12\left(4^m+\binom{2m}{m}\right).
\end{equation}
\end{theorem}

\begin{proof}
By Lemma~\ref{lem:cyclic-defect-decomposition} and symmetry,
\[
 P_n^{\mathrm{cyc}}
 =P_{n,m}^{\mathrm{cyc}}
  +2\sum_{r=0}^{m-1}P_{n,r}^{\mathrm{cyc}}.
\]
For fixed $r$, the corresponding sequence has an annihilator of degree at
most $\binom{2m}{r}$.  Multiplying these annihilators for $0\le r\le m$
therefore gives an annihilator for the displayed sum of degree at most
\[
 \sum_{r=0}^{m}\binom{2m}{r}
 =\frac12\left(4^m+\binom{2m}{m}\right),
\]
as claimed.
\end{proof}

\begin{proposition}[Skew-symmetric cyclic permanents]\label{prop:per-skew-cyclic}
Assume characteristic different from $2$ and impose $a_0=0$ and
$a_{-s}=-a_s$.  Then every odd-order skew-symmetric circulant permanent
vanishes.  The even-size stable subsequence has a homogeneous
constant-coefficient annihilator of degree at most
\begin{equation}\label{eq:per-skew-cyclic-even}
 2^{2m-1}.
\end{equation}
\end{proposition}

\begin{proof}
By \eqref{eq:cyclic-defect-skew-pair}, the sectors $r$ and $2m-r$ cancel
when $n$ is odd.  The central sector $r=m$ satisfies the same identity with
itself and therefore also vanishes at odd $n$.  Hence
$P_{2k+1}^{\mathrm{cyc}}=0$.

For even sizes,
\[
 P_{2k}^{\mathrm{cyc}}
 =P_{2k,m}^{\mathrm{cyc}}
  +2\sum_{r=0}^{m-1}P_{2k,r}^{\mathrm{cyc}}.
\]
For $r<m$, passing from $n$ to $2k$ replaces the clean-band transfer by its
square and does not increase its dimension, so the $r$th contribution has an
annihilator of degree at most $\binom{2m}{r}$.  The central sequence
$P_{n,m}^{\mathrm{cyc}}$ has an annihilator of degree at most
$\binom{2m}{m}$ and vanishes at every odd index.  After reducing its rational
generating function to coprime numerator and denominator, the same parity
argument as in Proposition~\ref{prop:per-skew-open} makes the denominator
even.  Thus its even subsequence has an annihilator of degree at most
$\frac12\binom{2m}{m}$.

Multiplying the annihilators of the independent sector contributions gives
\[
 \sum_{r=0}^{m-1}\binom{2m}{r}
 +\frac12\binom{2m}{m}
 =\frac{4^m}{2}=2^{2m-1},
\]
which proves the bound.
\end{proof}

The proved permanent annihilator bounds can be summarized without asserting
all-$m$ minimality:
\[
\begin{array}{c|cc}
 & \begin{array}{c}\text{symmetric}\\[-2pt]\text{full sequence}\end{array}
 & \begin{array}{c}\text{skew-symmetric}\\[-2pt]\text{even subsequence}\end{array}\\ \hline
 \text{open Toeplitz}
 & \dfrac{\binom{2m}{m}+2^m}{2}
 & \dfrac12\binom{2m}{m}\\[7pt]
 \text{circulant}
 & \dfrac{4^m+\binom{2m}{m}}{2}
 & 2^{2m-1}
\end{array}
\]
The two symmetric bounds are natural candidates for generic sharpness.  Their
minimality in higher semibandwidth is a separate problem and is not pursued
here.  The skew column records only the proved even-subsequence upper bounds.

\section{Concluding remarks}\label{sec:conclusion}

The preceding companion paper \cite{AlekseyevKhomovskyRecurrences2026}
isolates the universal recurrence mechanism for fixed-band Toeplitz
determinants and permanents.  The present follow-up shows that additional
structure acts in three different ways: reciprocal-root symmetry collapses
determinant modes, exceptional support patterns convert certain permanents
exactly into determinants, and involutions of permanent state spaces reduce
the observable transfer quotient.  Keeping these mechanisms separate makes
clear which recurrence reductions are spectral, which are exact identities,
and which come from finite-state symmetry.

For determinants the root-product picture remains the most direct spectral
route, but both open symmetry classes now have parallel constructive transfer
explanations.  In the symmetric family, determinant straightening reduces the
boundary-minor state space to the primitive symplectic module of Catalan
dimension $C_{m+1}$, whose weight multiplicities collapse for a fixed
autonomous transfer to $(3^m+1)/2$ distinct eigenvalues.  In the skew family,
fixed-ambient complementary minors obey transpose-sign relations whose
half-binomial count mirrors the canonical middle Hodge decomposition, but the
normalized row-column transfer does not realize that decomposition by a
levelwise transpose quotient.  In the compound realization the one-step
transfer exchanges the two Hodge halves, and the two-step transfer has
$3^{m-1}$ generic ternary modes, yielding the full $2\cdot3^{m-1}$ bound.
Widom gives the same modes directly, and the
Hankel--Vandermonde argument proves that the principal determinant coordinate
generically sees all of them.  A minimal circulant completion provides a
complementary bridge between the open and cyclic determinant problems: an
$n\times n$ Toeplitz determinant equals the determinant of an
$(n+m)\times(n+m)$ circulant times an $m\times m$ inverse-block correction,
and expanding that correction gives a circulant-completion counterpart to
the classical Baxter--Schmidt derivation of the Widom formula.  The skew
Toeplitz factorization also provides a direct matrix explanation of the square
structure at even size.  Under cyclic closure the full Boolean subset
spectrum reduces from generic degree $4^m$ to $3^m$ in the symmetric balanced
family, while a skew-symmetric circulant determinant vanishes identically.

Exact permanent--determinant conversion is instead a low-semibandwidth
phenomenon.  The zero-diagonal pentadiagonal support admits a fixed
P\'olya--Kasteleyn converter, and on the Toeplitz locus its real signing is
gauge-equivalent to the corresponding fourth-root-of-unity phase rule.  The
paired fixed-point renewal identities reconstruct every restored-main-diagonal
layer of both the permanent and determinant from the corresponding
zero-diagonal sequence.  At the full coordinate-support level, a
six-matching obstruction at order six shows that arbitrary nonzero entrywise
multipliers cannot convert the $(3,2)$ band; specialization then excludes all
larger consecutive two-sided supports.  Thus the support-level classification
leaves only the Hessenberg families and the pentadiagonal exception.  The
separate diagonal-wise classification on the Toeplitz locus yields exactly
the same list.

For permanents in arbitrary balanced semibandwidth, transposition gives the
open symmetric quotient.  Under cyclic closure, a column shift collects the
wrap-around entries into a fixed corner defect; resolving that defect produces
shifted clean bands $(2m-r,r)$, and transposition pairs the sectors $r$ and
$2m-r$.  This yields the open and cyclic symmetric bounds
\[
 \frac12\left(\binom{2m}{m}+2^m\right),
 \qquad
 \frac12\left(4^m+\binom{2m}{m}\right),
\]
while skew-symmetry forces odd-order vanishing and corresponding bounds for
the even subsequences.  These conclusions use only finite-state symmetries
and therefore hold uniformly in $m$.  The symmetric pentadiagonal open case
shows that the open quotient bound can be attained generically and also
exhibits an explicit one-step degree-drop locus.  Generic minimality of the
higher-semibandwidth permanent bounds is a separate problem and is not
required for the results established here.

Taken together, the results give a compact symmetry taxonomy for Toeplitz
recurrences.  For symmetric determinants, Laplace signs first impose universal
straightening relations among boundary minors,
\[
 \binom{2m}{m}\longrightarrow C_{m+1},
\]
and the primitive symplectic realization explains the second autonomous
compression
\[
 C_{m+1}\longrightarrow\frac{3^m+1}{2}
\]
as the passage from Catalan weight multiplicities to one eigenvalue per
admissible ternary weight.  For skew determinants, the canonical middle Hodge splitting in the
compound realization gives
\[
 \binom{2m}{m}\longrightarrow\frac12\binom{2m}{m},
\]
with fixed-ambient transpose parity yielding the same half-binomial coordinate
count; the two-step autonomous spectrum then compresses further to $3^{m-1}$
ternary weights, corresponding to $2\cdot3^{m-1}$ one-step modes.  Widom reaches the same final modes directly
and, through the Hankel criterion, proves scalar generic sharpness.
Exceptional low-width supports can identify permanent and determinant
sequences, while permanent symmetries act directly on boundary-state
quotients.  Thus state-space, representation-theoretic, and spectral
reductions are complementary rather than competing explanations.

\section*{Acknowledgements}
The authors used ChatGPT (OpenAI) as a research and writing aid for literature
searches, exploratory symbolic and computational work, code development and
checking, and language editing.  AI-assisted suggestions were treated as
provisional and checked against direct proofs, independent computations, or
cited sources before inclusion.  The authors take full responsibility for the
mathematical content of the article.

\printbibliography
\end{document}